\documentclass[A4paper,10pt]{ctexart}
\usepackage{amsmath}
\usepackage{amssymb}
\usepackage{amsthm}
\usepackage{xltxtra}
\usepackage{CJKutf8}
\usepackage{texnames}
\usepackage{mflogo}
\usepackage{bm}
\usepackage{lmodern}
\usepackage{amscd}
\usepackage[all]{xy}
\usepackage[numbers,sort&compress]{natbib}
\usepackage{tikz,tikz-cd}
\usepackage{sseq}
\usepackage{cancel}
\graphicspath{{figures/}}

\usepackage{titlesec}
\usepackage{titling}

\pretitle{\begin{center}\LARGE\bfseries}
	\posttitle{\end{center}\vspace{\baselineskip}}

\usepackage[paperwidth=8.3in, paperheight=11.7in, margin=1in]{geometry}

\titleformat{\section}[block]{\Large\bfseries}{\thesection}{1em}{}[\titlerule]
\titleformat{\subsection}[block]{\large\bfseries}{\thesubsection}{1em}{}
\titleformat{\subsubsection}[runin]{\large\bfseries}{\thesubsubsection}{1em}{}[.]

\usepackage{setspace}
\usepackage{multicol}
\newtheorem{definition}{Definition}[section]
\newtheorem{theorem}{Theorem}[section]
\newtheorem{corollary}{Corollary}[section]
\newtheorem{lemma}{Lemma}[section]
\newtheorem{example}{Example}
\newtheorem{remark}{Remark}
\newtheorem{proposition}{Proposition}[section]
\newenvironment{acknowledgements}{
	\par\medskip
	\noindent\textit{Acknowledgements.}\
}{\par\medskip}
\newtheorem{claim}{Claim}
\newtheorem{problem}{Problem}[section]

\newcommand{\R}{\mathbb{R}}
\newcommand{\C}{\mathbb{C}}

\newcommand{\ii}{i}
\newcommand{\Z}{\mathbb{Z}}
\newcommand{\Qua}{\mathbb{H}}
\newcommand{\A}{\mathcal{A}}
\newcommand{\G}{\mathcal{G}}
\newcommand{\M}{\mathcal{M}}
\newcommand{\abs}[1]{\left|#1\right|}

\DeclareMathOperator{\ind}{ind}
\DeclareMathOperator{\Res}{Res}

\ctexset{
	subsection={
		format+ =\zihao{4} \heiti,
		name = {,\quad},
		number = \arabic{section}.\arabic{subsection},
		aftername=\hspace{1pt},
	}
}

\title{$Pin^{-}(2)$ Bauer--Furuta invariants}
\author{Hao Wu}
\date{}
\begin{document}
	\maketitle
	\begin{abstract}
		Adapting Bauer and Furuta's constructions of the refinement of the Seiberg--Witten invariants, we establish the analogous stable cohomotopy refinement of the $Pin^{-}(2)$ monopole invariants proposed by Nakamura \cite{nakamura2015pin}, and give the corresponding connected sum formula. 
		%This thesis provides a partial answer to remark 2.13 in \cite{nakamura2015pin}.  
	\end{abstract}
	\tableofcontents

	\section{Introduction}	
	
	In the mid-1990s, Seiberg and Witten introduced new invariants \cite{witten1994monopoles} to give new insights into the differential topology of smooth four-manifolds. The Seiberg--Witten invariants are defined by a kind of oriented counting of the gauge equivalence classes of solutions to a system of equations. Specifically, for an oriented closed Riemannian four-manifold $(X,g)$, we equip it with a $Spin^c$-structure $\mathfrak{s}$, which is given by a principal $Spin^c(4)$ bundle $P\to X$ lifting the principal frame bundle $Fr(X)\to X$, and $P\to Fr(X)$ is the adjoint representation $\rho: Spin^c(4)\to SO(4)$ on each fiber. Forming the spinor bundles $S^{\pm}$ and the determinant line bunle $L$ asscociated to $P$, we can consider the space of $S^1$ connections $\A(L)$ on $L$ and the spaces of smooth sections $\Gamma(S^{\pm})$ of $S^{\pm}$. Choose $A\in\A(L)$ and $\phi\in\Gamma(S^+)$, after lifting the connection $A$ to a $Spin^c(4)$-connection on $P$ which we still denote by $A$, we can define the so-called Seiberg-Witten equations
    \begin{equation}\label{Seiberg-Witten equations}
	    \begin{cases}
	    	D_A\phi=0\\
	        F_A^+=q(\phi),
	    \end{cases}
    \end{equation}
    for $(A,\phi)\in\A(L)\times\Gamma(S^+)=\mathcal{C}$, where $D_A$ is the Dirac operator associated to $A$ on $S^+$, $F_A^+$ is the self-dual part of the curvature term $F_A$, and $q(\phi)$ is a specific quadratic expression. Introducing the space of gauge transformations $\G$ on $P$ which cover the identity map on $Fr(X)$, we find that $\G$ is isomrphic to $\operatorname{Map}(X,S^1)$. The action is free on $\mathcal{C}^*=\A(L)\times(\Gamma(S^+)\setminus\{0\})$, and the space $\mathcal{B}^*=\mathcal{C}^*/{\G}$ is homotopy equivalent to $\C P^{\infty}\times T^{b_1(X)}$. It turns out that the equations \eqref{Seiberg-Witten equations} are invariant under the action of $\G$. The space of all solutions to \eqref{Seiberg-Witten equations} modulo gauge transformations, denoted by $\M(X,\mathfrak{s})$, is called the moduli space of monopoles. Under certain genericity assumptions, the moduli space is a smooth, orientable finite dimensional manifold, defining a homology class $[\M(X,\mathfrak{s})]$ in $\mathcal{B}^*$. The dimension of the moduli space, which we denote by $d(\mathfrak{s})$, is computed through the index theorem as $d(\mathfrak{s})=2\operatorname{ind}_{\C}(D_A)-\left(1-b_1(X)+b^+(X)\right)=\frac{c_1(L)^2-2\chi(X)-3\sigma(X)}{4}$. The Seiberg--Witten invariant is the characteristic number obtained from these data by specifying an orientation of the moduli space:
    
    \begin{equation*}
    	 SW(X,\mathfrak{s})=
    	 \begin{cases}
    	 	  \int_{\M(X,\mathfrak{s})} \xi^{\frac{d(\mathfrak{s})}{2}}, &\text{if $d(\mathfrak{s})$ is even and nonnegative},\\
    	 	  0, &\text{otherwise}
    	 \end{cases}
    \end{equation*}
    where $\xi\in H^2(\mathcal{B}^*)\cong H^2\left(\C P^{\infty}\times T^{b_1(X)}\right)$ is the generator for the $\C P^{\infty}$ factor.
    
    Later, using ideas of finite dimensional approximation, Bauer and Furuta introduced a stable cohomotopy refinement of Seiberg--Witten invariants \cite{bauerfuruta2004stable,bauer2004stable}, known as the Bauer--Furuta-invariants. Instead of taking the numerical oriented count of orbits of the solutions, they considered the monopole map itself. Fix a base point $*\in X$, define the pointed gauge group $\G_0=\{\ker(\operatorname{Map}(X,S^1)\to S^1)\}$ as the kernel of the evaluation map at $*$, then the monopole map 
    \begin{align*}
    	\mu: (A+\ker d)\times (\Gamma(S^+)\oplus\Omega^1(X)\oplus H^0(X))/{\G_0} &\to (A+\ker d)\times (\Gamma(S^-)\oplus\Omega^+(X)\oplus\Omega^0(X)\oplus H^1(X))/{\G_0}\\
    	(A,\phi,a,f) &\mapsto (A,D_{A+a}\phi,F_{A+a}^+-q(\phi),d^*a+f,pr(a))
    \end{align*}
    is a fiber-preserving, $S^1$-equivariant map over $\mathrm{Pic}^{\mathfrak{s}}(X)$. It turns out that $\mu$ defines an element in the equivariant stable cohomotopy group $\pi^{b^+}_{S^1,H}(\mathrm{Pic}^{\mathfrak{s}}(X);\operatorname{ind}(D))$. This invariant avoids the delicate transversality problems in the standard Seiberg-Witten theory. After establishing the connected sum formula $[\mu_{X}]=[\mu_{X_0}]\wedge[\mu_{X_1}]$ for $X=X_0\# X_1$, Bauer proved that this cohomotopy invariant actually encodes more information than the Seiberg-Witten invariant and it has led to new applications in $4$-manifold topology.
    
    Building on ideas of Bauer and Furuta, Manolescu \cite{manolescu2003seiberg} developed the Seiberg--Witten--Floer stable homotopy type $SWF(Y,\mathfrak{c})$ for a closed, oriented $3$-manifold $(Y,\mathfrak{c})$ with $b_1(Y)=0$ and a given $Spin^c$-structure $\mathfrak{c}$, and introduced a relative Bauer-Furuta invariant for $4$-manifolds with boundary. Later, the gluing formula \cite{manolescu2007gluing} was established for this relative invariant.

	Inspried by Fr\o yshov  \cite{froyshov20124}, in contrast to refining the Seiberg-Witten theory, Nakamura \cite{nakamura2013monopole,nakamura2015pin} introduced a variant of the Seiberg--Witten equations. Instead of lifting the structure group $SO(4)$ of the principal frame bundle to $Spin^c(4)$, he considered lifting to $Spin^{c-}(4)=Spin(4)\times_{\{\pm 1\}} Pin^{-}(2)$ and defined the analogous $Pin^{-}(2)$ monopole invariants for a double cover $\tilde{X}\to X$ with associated free involution $\iota:\tilde{X}\to\tilde{X}$. The $Pin^{-}(2)$ monopole invariants for $X$ share many similarities with the ordinary Seiberg--Witten invariants, and can be explicitly computed for some K$\ddot{a}$hler surfaces and symplectic $4$-manifolds \cite{nakamura2020real}. Later, Konno, Taniguchi, Miyazawa \cite{konno2024involutions,miyazawa2023gauge} defined an invariant for a triple $(\tilde{X},\mathfrak{s},\iota)$, where $\iota$ is a nonfree involution on the oriented, closed $4$-manifold $\tilde{X}$ with nonempty, $2$-dimensional fixed point set $\tilde{X}^{\iota}$. It is therefore natural to ask whether such an invariant exists in the case $\tilde{X}^{\iota}$ is empty.

	Our work is a combination of the Bauer--Furuta theory and the $Pin^{-}(2)$ monopole theory. By modifying the argument of Bauer--Furuta \cite{bauerfuruta2004stable}, we obtain the following analogous stable cohomotopy refinement of the $Pin^{-}(2)$ monopoles invariants:
	
	\begin{theorem}\label{main theorem1}
		The monopole map $\mu:\A\to \mathcal{C}$ defines an element in an equivariant stable cohomotopy group
		$$\pi_{\Z_2,H}^{b^+(X;l)}(\mathrm{Pic}^c(X);\ind(D)),$$
		which is independent of the chosen Riemannian metric. For $b^+(X;l)>\dim(Pic^c(X))+1$, we have a natural homomorphism of the cohomotopy group to $\Z_2$, which maps $[\mu]$ to the $\Z_2$-valued $Pin^{-}(2)$ monopole invariant.
	\end{theorem}
	
	Adapting the argument of Bauer \cite{bauer2004stable}, one version of the connected sum formulas is obtianed as the following:
	
	\begin{theorem}\label{main theorem2}
		Let $X_1$ be a closed oriented connected $4$-manifolds with twisted $Spin^{c-}$-structure $c_1$ with $b_1(X_1;l_1)=0$. Let $X_2$ be a closed oriented connected $4$-manifolds with a $Spin^{c}$-structure $c_2$\textnormal{(untwisted $Spin^{c-}$-structure)} with $b_1(X_2)=0$. Let $X=X_1\# X_2$ have the connected sum $Spin^{c-}$-structure $c_1\# c_2$. Then the stable equivariant cohomotopy invariant is the smash product of the invariants of its summands
		$$[\mu_{X}]=[\mu_{X_1}]\wedge \Res^{S^1}_{\Z_2}[\mu_{X_2}].$$
	\end{theorem}
	
	These two theorems provide new insights into the differential structures of smooth $4$-manifolds.
	
	The structure of the thesis is organized as follows: In section \ref{Preliminaries}, we review necessary background materials on $Pin^{-}(2)$-monopole theory, following \cite{nakamura2013monopole,nakamura2015pin}. In section \ref{The monopole map and the stable cohomotopy invariant}, we develop the stable cohomotopy refinement of $Pin^{-}(2)$-monopole theory, and give the proof of theorem \ref{main theorem1}. As examples, we give computations for the cohomotopy invariants of $N$ and several other elliptic surfaces. In section \ref{Connected sum formula}, we adapt the method of \cite{bauer2004stable} to proving the connected sum formula, and propose the computation of the cohomotopy invariant of $N\#K$ as an application. In section \ref{Problems and Future Directions}, we propose some further tasks which grow out of our work. 
	
    \begin{acknowledgements}
    	The author would like to express deep gratitude to his advisor Professor Tsuyoshi Kato for consistent encouragement and insightful guidance during the preparation of this thesis. He also wishes to thank professor Nobuhiro Nakamura for introducing this field, for invaluable discussions, and for providing numerous helpful references, without whom the completion of this thesis would not be so smooth.
    \end{acknowledgements}
	
	\section{Preliminaries}\label{Preliminaries}
	
	A large portion of this section is due to \cite{nakamura2013monopole,nakamura2015pin}.
	
	\subsection{$Spin^{c-}$ structures}
	\begin{definition} [$Spin^{c-}$ {\bfseries groups}] Let $Pin^{-}(2)$ be the subgroup of $Sp(1)$ generated by $U(1)$ and $j$, i.e., $Pin^{-}(2)=U(1)\cup jU(1)$. There exists a $2$-to-$1$ homomorphism $\phi_0:Pin^{-}(2)\to O(2)$ sending $z$ to $z^2$ and $j$ to 
		$$\begin{pmatrix}
			1 & 0 \\
			0 & -1
		\end{pmatrix}.$$
	We define the group $Spin^{c-}(n)$ as $Spin(n)\times_{\{\pm 1\}} Pin^{-}(2)$. There is an exact sequence
	$$1\to \{\pm 1\}\to Spin^{c-}(n)\to SO(n)\times O(2)\to 1.$$
	\end{definition}
	
	Let $X$ be an closed oriented connected Riemannian $n$ manifold with double cover $\tilde{X}\to X$. Denote the principal frame bundle on $X$ by $Fr(X)$. $Spin^c(n)$ is the identity component of $Spin^{c-}(n)$ and $Spin^{c-}(n)/{Spin^{c}(n)}=\{\pm 1\}$. Also we have $Spin^{c-}(n)/{Pin^{-}(2)}=SO(n)$ and $Spin^{c-}(n)/{Spin(n)}=O(2)$.
	
	\begin{definition}
	Let $\tilde{X}\to X$ be a double cover of a closed connected and oriented smooth $4$ manifold X. Then we can define a {\bfseries $Spin^{c-}$ structure} $(P,\sigma,\tau)$ where
		\begin{enumerate}
			\item $P\to X$ is a principal $Spin^{c-}(4)$ bundle over $X$,
			\item $\sigma:P/{Spin^c(4)}\to \tilde{X}$ is an isomorphism,
			\item $\tau:P/{Pin^{-}(2)}\to Fr(X)$ is an isomorphism between $SO(4)$ bundles.
		\end{enumerate}
	\end{definition}
	
	\begin{definition}
		The {\bfseries characteristic} $O(2)$-{\bfseries bundle} is the $O(2)$ bundle $E=P/{Spin(4)}$ associated to the $Spin^{c-}$ structure. Let $l$ be the $\Z$-bundle $\tilde{X}\times_{\{\pm 1\}}\Z$ over $X$, then $l$ is related to $E$ by $\det E=l\otimes\R=\lambda$.
	\end{definition}
	
	\begin{proposition}
		
		\begin{itemize}
			\item For an $O(2)$-bundle $E$ with $\det E=l\otimes\R$ as above, there exists a $Spin^{c-}$ structure on the double cover $\tilde{X}\to X$ with characteristic bundle $E$ if and only if $w_2(x)=w_2(E)+w_1(l\otimes\R)^2$.
			\item Given a $Spin^{c-}$ structure on $\tilde{X}\to X$, then the set of isomorphism class of $Spin^{c-}$ structures on $\tilde{X}\to X$ is in $1$-$1$ correspondence with the set $H^2(X;l)$.
		\end{itemize}
	\end{proposition}
	
	\begin{definition}
		If the cover $\tilde{X}\to X$ is nontrivial, we will say the $Spin^{c-}$ structure is {\bfseries twisted}. Otherwise when $\tilde{X}\to X$ is the trivial disconnected double cover, we will say the $Spin^{c-}$ structure is {\bfseries untwisted}, and this structure has a $Spin^c(4)$ reduction which induces a $Spin^c$ structure on $X$.
	\end{definition}
	
	\subsection{Spinor bundles and Clifford multiplication}
	
	Let $\Qua_T$ be a $Spin^{c-}(4)$ module which is isomorphic to $\Qua$ as a vector space, and such that $[q_+,q_-,u]\in Spin^{c-}(4)=(Sp(1)\times Sp(1))\times_{\{\pm 1\}} Pin^{-}(2)$ acts on $v\in\Qua_T$ by $q_+vq_-^{-1}$. Then we can identify $P\times_{Spin^{c-}(4)}\Qua_T$ with the tangent bundle $TX$.
	
	Let $\overline{\phi}:Spin^{c-}(4)\to O(2)$ be the homomorphism obtained from $\phi_0:Pin^{-}(2)\to O(2)$, then we can recover the bundle $E$ as $P\times_{\overline{\phi}}O(2)$.
	
	\begin{definition}
		Let $\Qua_{\pm}$ be $Spin^{c-}(4)$ modules which are copies of $\Qua$ as vector spaces, such that $[q_+,q_-,u]\in Spin^{c-}(4)$ acts on $\phi\in\Qua_{\pm}$ by $q_{\pm}\phi u^{-1}$. Then the {\bfseries positive} and {\bfseries negative spinor bundles} for the $Spin^{c-}$ structure are defined as $S^+=P\times_{Spin^{c-}(4)}\Qua_+$ and $S^-=P\times_{Spin^{c-}(4)}\Qua_-$ respectively, and we write $S=S^+\oplus S^-$.
	\end{definition}   
	
	\begin{definition}[Clifford multiplication $1$]
	The {\bfseries Clifford multiplication} $\rho_{\R}:\Omega^1(X)\times\Gamma(S^+)\to\Gamma(S^-)$ is defined via the $Spin^{c-}(4)$ equivariant map 
		\begin{align*}
			\Qua_T\times \Qua_+&\to \Qua_-\\
			(v,\phi)&\mapsto \bar{v}\phi.
		\end{align*}
	\end{definition}
	
	We will need a twisted complex version of the Clifford multiplication defined in the following context. Let $G_0\cong Spin^{c}(4)\subset Spin^{c-}(4)$ be the identity component, then $Spin^{c-}(4)/{G_0}\cong\{\pm 1\}.$ Let $\varepsilon:Spin^{c-}(4)\to Spin^{c-}(4)/{G_0}$ be the natural projection and let $Spin^{c-}(4)/{G_0}$ act on $\C$ by complex conjugation. Then $Spin^{c-}(4)$ acts on $\C$ through $\varepsilon$ and complex conjugation. Define the $Spin^{c-}(4)$ equivariant map $\rho_0$ as 
	\begin{align*}
		\rho_0:\Qua_T\otimes_{\R}\C\times\Qua_+&\to\Qua_-\\
		(v\otimes a,\phi)&\mapsto \bar{v}\phi\bar{a}.
	\end{align*}
	Define the bundle $K$ over $X$ by $K:=\tilde{X}\times_{\{\pm 1\}}\C$ where $\{\pm 1\}$ acts by complex conjugation. Note that $K=\underline{\R}\oplus\ii\lambda$. We can define the full Clifford multiplication via $\rho_0$ 
	$$\rho:\Omega^1(X;K)\times\Gamma(S^+)\to\Gamma(S^-).$$
	Restricting $\rho$ to $\underline{\R}$, we can recover $\rho_{\R}$.
	
	\begin{definition}[Clifford multiplication $2$]
		Restricting $\rho$ to $\ii\lambda$, we obtain another {\bfseries Clifford multiplication}
		$$\rho:\Omega^1(X;\ii\lambda)\times\Gamma(S^+)\to\Gamma(S^-).$$
	\end{definition}
    
    \subsection{Dirac operators}
    
    Suppose in the rest of this section that $\lambda$ is a nontrivial bundle and fix a Riemannian metric on $X$. Given an $O(2)$-connection $A\in\mathcal{A}(E)$ on $E$, the Levi-Civita connection and $A$ will induce a unique $Spin^{c-}(4)$-connection on $P$.
    \begin{definition}
   	    The {\bfseries Dirac operator} $\mathcal{D}_A$ associated to the $O(2)$-connection is defined via $\rho$ upon choosing a synchronous orthonormal frame $\{e_i\}$ as 
   	    \begin{align*}
   		\mathcal{D}_A:\Gamma(S)&\to\Gamma(S)\\
   		\phi&\mapsto\sum_{i=1}^{4}\rho(e_i)\nabla_{A,e_i}\phi.
   	    \end{align*}
   	    And there is an induced operator $\mathcal{D}_A^+:\Gamma(S^+)\to\Gamma(S^-)$ defined by the same formula.
    \end{definition} 
	The definition of this Dirac operator $\mathcal{D}_A$ is independent of the choice of the frame and share similar properties with the ordinary ones. If $A'\in\mathcal{A}(E)$ is another $O(2)$ connection on $E$, then the difference $a=A'-A$ is in $\Omega^1(X;\ii\lambda)$ and $\mathcal{D}_{A'}$ and $\mathcal{D}_A$ are related via $\rho$ by
	$$\mathcal{D}_{A+a}\phi=\mathcal{D}_A\phi+\rho(a)\phi.$$
	
	There are {\bfseries twisted} hermitian inner products defined on spinor bundles
	\begin{equation*}
		\langle\cdot,\cdot\rangle_{K,x}:S^{\pm}_x\times S^{\pm}_x\to K_x. \rlap{\quad (*)}
	\end{equation*}
	The real part of $(*): \langle\cdot,\cdot\rangle_{\R,x}=\text{Re}\langle\cdot,\cdot\rangle_{K,x}$ defines a real inner product on $S^{\pm}$. Then in fact the Dirac operator is formally self adjoint with respect to the $L^2$ inner product induced from $\langle\cdot,\cdot\rangle_{\R}$
	
	\begin{proposition}
		$\mathcal{D}_A$ is formally self adjoint in the sense that
		$$(\phi,\mathcal{D}_A\psi)_{L^2}=(\mathcal{D}_A\phi,\psi)_{L^2},$$
		where $(\phi_1,\phi_2)_{L^2}=\int_X \langle\phi_1,\phi_2\rangle_{\R}dvol.$
	\end{proposition}
	We will write $D_A=\mathcal{D}_A^+$ when no confusion arises in subsequent sections.

	\subsection{$Pin^{-}(2)$ monopole equations}
	
	The curvature $F_A$ of $A$ belongs to $\Omega^2(X;\ii\lambda)$. $\Omega^+(X;\ii\lambda)$ is associated to $P$ as follows. Let $\varepsilon:Pin^-(2)\to Pin^-(2)/{U(1)}\cong \{\pm 1\}$ be the projection and let $[q_+,q_-,u]\in Spin^{c-}(4)$ act on $v\in \operatorname{Im}\Qua$ by $\varepsilon(u)q_+vq_+^{-1}$. Then $\Gamma(P\times_{Spin^{c-}(4)}\operatorname{Im}\Qua)$ is isomorphic to $\Omega^+(X;\ii\lambda).$ For $\phi\in\Qua_+$, $Spin^{c-}(4)$ acts on $\phi\ii\overline{\phi}\in \operatorname{Im}\Qua$ similarly. So we can define a quadratic map
	\begin{align*}
		q:\Gamma(S^+) &\to \Omega^+(X;\ii\lambda)\\
		\phi &\mapsto\phi\ii\overline{\phi}.
	\end{align*}
	
	\begin{definition}
		Let $\mathcal{A}(E)$ be the space of $O(2)$-connections on $E$. Then the $Pin^{-}(2)$ {\bfseries monopole equations} for the pair $(A,\phi)\in \mathcal{A}(E)\times \Gamma(S^+)$ are defined by
		\begin{equation}\label{monopole equations}
			\begin{cases}
				D_A\phi=0,\\
				F_A^+=q(\phi).
			\end{cases}
		\end{equation}
	\end{definition}
	
	Let $\mathcal{C}=\mathcal{A}(E)\times \Gamma(S^+)$ denote the configuration space and $\mathcal{C}^*=\mathcal{A}(E)\times (\Gamma(S^+)\setminus\{0\})$ denote the space of irreducible configurations. Fix $k\geq 4$, and take $L_k^2$ completions of $\mathcal{C}$. The $Pin^{-}(2)$ monopole equations are assumed to be equations for $L_k^2$ configurations.
	
	\begin{definition}
		The {\bfseries gauge group} $\G$ is identified with $\Gamma(\tilde{X}\times_{\{\pm 1\}}U(1))$, where $\pm 1$ act on $U(1)$  by complex conjugation.
	\end{definition}
	
	We take $L_{k+1}^2$ completion of $\G$. The gauge action of $g\in\G$ on $\mathcal{A}(E)\times \Gamma(S^+)$ is given by $g(A,\phi)=(A-2g^{-1}dg,g\phi)$. If $\phi\neq 0$, then $\G$-action on $(A,\phi)$ is free and we call $(A,\phi)$ an {\bfseries irreducible}. On the other hand, $(A,0)$ is a {\bfseries reducible} with the stablizer isomorphic to $\{\pm 1\}$, which is the subgroup of constant sections $\{\pm 1\}\subset\G.$

	\subsection{$Pin^{-}(2)$ monopole invariants}
	
	\begin{definition}
		The {\bfseries moduli space} $\M(X,c)=\M_{Pin^{-}(2)}(X,c)$ is defined as the space of solutions modulo gauge transformations. \textnormal{(The perturbed moduli space is usually denoted by the same symbol)} 
	\end{definition}
	
	\begin{proposition}
		Let $\mathcal{B}^*=\mathcal{C}^*/{\G}$, then $\mathcal{B}^*$ has the homotopy type of $\R P^{\infty}\times T^{b_1(X;l)},$ and $\G\simeq \Z_2\times \Z^{b_1(X:l)}.$
	\end{proposition}
	
	Since the $\G$ action is not free on $\mathcal{A}(E)$, we will give a subgroup of $\G$ which acts on $\mathcal{A}(E)$ freely. Take a loop $\gamma:S^1\to X$ such that the restrction $\lambda|_{\gamma}=\gamma^*\lambda$ is a nontrivial $\R$ bundle over $\gamma$. Let $\tilde{\gamma}\to\gamma$ be the connected double cover of $\gamma$ and define $\G_{\gamma}$ by $\G_{\gamma}=\Gamma(\tilde{\gamma}\times_{\{\pm 1\}}U(1))$. Then $\G_{\gamma}$ satisfies:
	\begin{itemize}
		\item We have an epimorphism $\G\to\G_{\gamma}$ by restricting $\G$ to $\gamma$.
		\item $\pi_0\G_{\gamma}\cong \{\pm 1\}$.
		\item By restriction and projection, we have an epimorphism 
		$$\theta_{\gamma}:\G\to \pi_0\G_{\gamma}\cong \{\pm 1\}.$$
	\end{itemize}
	
	Define $\mathcal{K}_{\gamma}=\ker\theta_{\gamma}.$ Consider the exact sequence $$1\to \{\pm 1\}\to \G\to \G/{\{\pm 1\}}\to 1.$$
	Then $\G/{\{\pm 1\}}$ is homotopy equivalent to $\Z^{b_1(X;l)}$, and $\theta_{\gamma}$ gives a splitting of this sequence. We will use the group $\mathcal{K}_{\gamma}$ in section \ref{The monopole map and the stable cohomotopy invariant} in constructing the stable cohomotopy invariant.
	
	\begin{theorem}
		Suppose $b^+(X;l)\geq 1$, then by a generic choice of perturbation $\eta\in\Omega^+(X;\ii\lambda)$, the moduli space $\M(X,c)$ has no reducible and is a compact manifold with dimension $d(c)$ given by
		$$d(c)=\frac{\tilde{c_1}(E)^2-\sigma(X)}{4}-(b_0(X;l)-b_1(X;l)+b^+(X;l)),$$
		where $ind(D_A)=\frac{\tilde{c_1}(E)^2-\sigma(X)}{4}$. 
	\end{theorem}
	
	Therefore by this theorem, the moduli space $\M(X,c)$ represents a class $[\M(X,c)]$ in $H_{d(c)}(\mathcal{B}^*)$. In general, the moduli space is not orientable, but we can defined the $\Z_2$ valued $Pin^{-}(2)$ monopole invariants.
	
	\begin{definition}
		The $Pin^{-}(2)$ {\bfseries monopole invariants} of $(X,c)$ is defined as 
		$$SW^{Pin}(X,c)=\langle\xi^{d(c)},[\M(X,c)]\rangle,$$
		where $[\M(X,c)]\in H_{d(c)}(\mathcal{B}^*;\Z_2)$ represents the fundamental class of $\M(X,c)$, and $\xi\in H^{1}(\mathcal{B}^*;\Z_2)$ represents the image of the generator $\zeta\in H^1(\R P^{\infty};\Z_2)$ under the restrcition map $\pi^*:H^1(\R P^{\infty};\Z_2)\to H^1(\R P^{\infty}\times T^{b_1(X;l)};\Z_2)$ induced by the projection $\pi:\mathcal{B}^*\simeq \R P^{\infty}\times T^{b_1(X;l)}\to \R P^{\infty}$.
	\end{definition}
	
	If $b^+(X;l)\geq 2$, then $SW^{Pin}(X,c)$ is a diffeomorphism invariant. If $b^+(X;l)=1$, then $SW^{Pin}(X,c)$ depends on the chamber structure and there is a wall-crossing formula.
	
	\subsection{A gluing formula}
	
	We only demonstrate one version of the gluing theorems in \cite{nakamura2015pin}.
	
	\begin{theorem}\label{Pin(2) gluing formula}
		Let $X_1$ be a closed oriented connected $4$-manifold with a twisted $Spin^{c-}$ structure $c_1$ with $b_1(X_1;l_1)\geq 1$. Let $X_2$ be a closed oriented connected $4$-manifold with a \textnormal{(twisted or untwisted)} $Spin^{c-}$ structure $c_2$, and suppose one of the followings:
		\begin{itemize}
			\item $b^+(X_2)\geq 1$, and $c_2$ is an untwisted $Spin^{c-}$ structure on $X_2$.
			\item $c_2$ is a twisted $Spin^{c-}$ structure on $X_2$ with $b^+(X_2;l_2)\geq 1$.
		\end{itemize}
		Then $SW^{Pin}(X_1\#X_2,c_1\#c_2)=0$.
	\end{theorem}

	\section{The monopole map and the stable cohomotopy invariant}\label{The monopole map and the stable cohomotopy invariant}
	
	Let $S^+$ and $S^-$ denote the real rank-4 bundles associated to the given $Spin^{c-}$ structure on $\tilde{X}\to X$ and Let $E$ be the associated $O(2)$-bundle over $X$.
	
	For a $Spin^{c-}$ connection $A$, denote by $D_A:\Gamma(S^+)\to \Gamma(S^-)$ the associated Dirac operator. We can define the monopole map $\tilde{\mu}$ on X by
	\begin{align*}
		\tilde{\mu}:\A(E)\times(\Gamma(S^+)\oplus\Omega^1(X;\ii\lambda)) \to \\
		\A(E)\times (\Gamma(S^-)\oplus\Omega^+(X;\ii\lambda)\oplus\Omega^0(X;\ii\lambda)\oplus H^1(X;\ii\lambda))\\
		(A,\phi,a) \to (A,D_{A+a}{\phi},F_{A+a}^+-q(\phi),d^*a,pr(a))
	\end{align*}
	
	The full gauge group $\G=\Gamma(\tilde{X}\times_{\Z_2} U(1))$ acts on spinors via multiplication with $g\in\G$, on connections via subtracting with $2g^{-1}dg$ and trivially on forms.
	
	Let $A$ be a fixed connection, then the subspace $A+\ker d \subset \A(E)$ is invariant under the free action of the subgroup $\mathcal{K}_{\gamma}\subset \G$ with quotient space isomorphic to 
	$$\mathrm{Pic}^c(X)=H^1(X;\lambda)/{\operatorname{Free} {H^1(X;l)}}.$$
	Let $\A$ and $\mathcal{C}$ denote the quotients 
	\begin{align*}
		&\A=(A+\ker d)\times (\Gamma(S^+)\oplus\Omega^1(X;\ii\lambda))/{\mathcal{K}_{\gamma}}\\
		&\mathcal{C}=(A+\ker d)\times (\Gamma(S^-)\oplus\Omega^+(X;\ii\lambda)\oplus\Omega^0(X;\ii\lambda)\oplus H^1(X;\ii\lambda))/{\mathcal{K}_{\gamma}}
	\end{align*}
	by $K_{\gamma}$. Both $\A$ and $\mathcal{C}$ are bundles over $\mathrm{Pic}^c(X)$ and the quotient
	$$\mu=\tilde{\mu}/{\mathcal{K}_{\gamma}}:\A\to \mathcal{C}$$
	of the monopole map is a fiber-preserving, $\Z_2$ equivariant map over $\mathrm{Pic}^c(X)$.
	
	%A finite dimensional approximation of the monopole map restricted to a fiber would be of the following form 
	%$$f:\R^m\oplus\tilde{\R}^{n+k}\to \R^{m+b}\oplus\tilde{\R}^{n},$$
	%where $\Z_2$ acts trivially on $\R$ and acts via the sign representation on $\tilde{\R}$.
	
	For fixed $k>4$, consider fiberwise $L_k^2$ Sobolev completion $\A_k$ of $\A$ and fiberwise $L_{k-1}^2$ Sobolev completion $\mathcal{C}_{k-1}$ of $\mathcal{C}$, then $\mu$ extends to a continuous map $\mu=\mu_k:\A_k\to \mathcal{C}_{k-1}$ over $\mathrm{Pic}^c(X)$.
	
	We can decompose $\mu$ into a sum $\mu=l+c$, where $l=(D_A,d^++d^*+pr)$ is a linear Fredholm map and $c:(\phi,a)\mapsto (0,F_A^+,0,0)+(a\cdot\phi,-q(\phi),0,0)$ is a compact map as the sum of a constant map, and composition of multiplication $\A_k\times \A_k\to \mathcal{C}_k$ and compact Sobolev embedding $\mathcal{C}_k\hookrightarrow \mathcal{C}_{k-1}$.
	
	\begin{proposition}\label{boundedness}
		Preimages $\mu^{-1}(B)\subset \A_k$ of bounded disk bundles $B\subset\mathcal{C}_{k-1}$ are contained in bounded disk bundles.
		\begin{proof}
			It suffices to prove this for the restriction of $\mu$ to $\{A\}\times (\Gamma(S^+)\oplus \ker d^*)$, which is a subspace of the fiber over the origin of the Picard torus, maps to $\{A\}\times (\Gamma(S^-)\oplus\Omega^+(X;\ii\lambda)\oplus H^1(X;\ii\lambda))$. We can do this restriction because $d^*: (\ker d^*)^{\bot}\to \Omega^0(X;\ii\lambda)$ is a linear isomorphism.
			
			Using the elliptic operator $D=D_A+d^+$ and its adjoint, define $L_k^2$ norms on respective function spaces via 
			$$(\cdot,\cdot)_i=(\cdot,\cdot)_0+(D\cdot,D\cdot)_{i-1}, \quad (\cdot,\cdot)=(\cdot,\cdot)_0=\int_X \langle\cdot,\cdot\rangle dvol.$$
			Define the $L_k^p$ norms correspondingly.
			
			Let $\mu(A,\phi,a)=(A,\varphi,b,pr(a))\in\mathcal{C}_{k-1}$ bounded by $R\in\R$, i.e. $\abs{(A,\varphi,b,pr(a))}\leq R$. The Weitzenb$\ddot{o}$ck formula for $D_{A'}=D_{A+a}$ states 
			$$D_{A'}^*D_{A'}\phi=\nabla_{A'}^*\nabla_{A'}\phi+\frac{1}{4}\kappa\phi+\frac{1}{2}F_{A'}^+\cdot\phi.$$
			Take pointwise inner product with $\phi$ we obtain
			\begin{align*}
				\Delta|{\phi}|^2 &=2\langle\nabla_{A'}^*\nabla_{A'}\phi,\phi\rangle-2\langle\nabla_{A'}\phi,\nabla_{A'}\phi\rangle\\
				&\leq 2\langle\nabla_{A'}^*\nabla_{A'}\phi,\phi\rangle\\
				&=2\langle D_{A'}^*D_{A'}\phi-\frac{1}{4}\kappa\phi-\frac{1}{2}F_{A'}^+\cdot\phi,\phi\rangle\\
				&=\langle 2D_{A'}^*\varphi-\frac{1}{2}\kappa\phi-(b+q(\phi))\phi,\phi\rangle 
			\end{align*}
			In particular, 
			\begin{align*}
				\Delta|{\phi}|^2 +\frac{1}{2}\kappa\abs{\phi}^2+\frac{1}{2}\abs{\phi}^4&\leq\langle 2D_A^*\varphi,\phi\rangle+\langle 2a\cdot\varphi,\phi\rangle-\langle b\phi,\phi\rangle\\
				&\leq 2\left(\|D_A^*\varphi\|_{L^{\infty}}+\|a\|_{L^{\infty}}\|\varphi\|_{L^{\infty}}\right)\cdot\abs{\phi}+\|b\|_{L^{\infty}}\cdot\abs{\phi}^2\\
				&\leq C_1\left((1+\|a\|_{L^{\infty}})\|\varphi\|_{L_{k-1}^2}\cdot\abs{\phi}+\|b\|_{L^{2}_{k-1}}\cdot\abs{\phi}^2\right),
			\end{align*}
			where in the last inequality we have used the embedding $L_{k-1}^2\hookrightarrow C^0, L_{k-2}^2\hookrightarrow C^0$ and $\|D_A^*\varphi\|_{L^{\infty}}\leq C_a\|D_A^*\varphi\|_{L^{2}_{k-2}}\leq C_aC_b\|\varphi\|_{L^2_{k-1}}.$
			
			Let $p>4$, then we have the estimate $\|a\|_{L^{\infty}}\leq C_2\|a\|_{L^{p}_1}$. Using elliptic estimate for the operator $D_A+d^+$ on $(0,a)$, we obtain
			$$\|a\|_{L_1^p}\leq C_3\left(\|d^+a\|_{L^p}+\|pr(a)\|\right).$$
			Combined with the equality $d^+a=b-F_A^++q(\phi)$, we obtain
			\begin{align*}
				\|a\|_{L^{\infty}}&\leq C_2\|a\|_{L^{p}_1}\\
				&\leq C_2C_3\left(\|d^+a\|_{L^p}+\|pr(a)\|\right)\\
				&\leq C_2C_3\left(\|pr(a)\|+\|b\|_{L^p}+\|F_A^+\|_{L^p}+\|q(\phi)\|_{L^p}\right)\\
				&\leq C_4\left(\|pr(a)\|+\|b\|_{L_{k-1}^2}+\|F_A^+\|_{L^p}+\|\phi\|^2_{L^{\infty}}\right)
			\end{align*}
			where we have used the embeddings $L_{k-1}^2\hookrightarrow C^0\hookrightarrow L^p$ and $L^{\infty}\hookrightarrow L^p.$
			
			Evaluating $|\phi|^2$ at a maximum point $x_0$ so that $\Delta|\phi(x_0)|^2\geq 0$ and we obtain 
			\begin{align*}
				\|\phi\|^4_{L^{\infty}}&\leq \left((1+\|a\|_{L^{\infty}})R\|\phi\|_{L^{\infty}}+R\|\phi\|^2_{L^{\infty}}\right)+\|\kappa\|_{L^{\infty}}\|\phi\|^2_{L^{\infty}}\\
				&\leq CR\left((1+\|a\|_{L^{\infty}}\right)\|\phi\|_{L^{\infty}}+\|\phi\|^2_{L^{\infty}})+\|\kappa\|_{L^{\infty}}\|\phi\|^2_{L^{\infty}}\\
				&\leq CR\left(\left(1+C_4(R+\|F_A^+\|_{L^p}+\|\phi\|^2_{L^{\infty}})\right)\|\phi\|_{L^{\infty}}+\|\phi\|^2_{L^{\infty}}\right)+\|\kappa\|_{L^{\infty}}\|\phi\|^2_{L^{\infty}}\\
				&\leq C_5R\left((1+R)\|\phi\|_{L^{\infty}}+\|\phi\|^2_{L^{\infty}}+\|\phi\|^3_{L^{\infty}}\right)+\|\kappa\|_{L^{\infty}}\|\phi\|^2_{L^{\infty}},		
			\end{align*} 
			which is a polynomial estimate for $\|\phi\|_{L^{\infty}}$. Therefore we obtain a $L^{\infty}$ estimate for $(\phi,a)$ and hence a fortiori for the $L^p$ norm of $(\phi,a)$ using the inequality $\|f\|_{L^p}\leq \operatorname{Vol}(X)^{\frac{1}{p}}\|f\|_{L^{\infty}}.$
			
			Now comes the bootstrapping: For $i\leq k$, assuming inductively $L_{i-1}^{2p}$-bounds on $(\phi,a)$ with $p=2^{k-i}$. To obtain $L_i^p$-bounds, we compute:
			\begin{align*}
				\|(\phi,a)\|^p_{L_i^p}-\|(\phi,a)\|^p_{L^p} &=\|(D_A\phi,d^+a)\|^p_{L_{i-1}^p}\\
				&=\|(\varphi-a\phi,b-F_A^++q(\phi))\|^p_{L_{i-1}^p}\\
				&\leq \|(\varphi,b)\|^p_{L_{i-1}^p}+\|(a\phi,-F_A^++q(\phi))\|^p_{L_{i-1}^p}\\
				&\leq \|(\varphi,b,pr(a))\|^p_{L_{i-1}^p}+\|(a\phi,-F_A^++q(\phi))\|^p_{L_{i-1}^p}.
			\end{align*}  
			Using Sobolev multiplication $L_{i-1}^{2p}\times L_{i-1}^{2p}\to L_{i-1}^{p}$, then the second summand in the last expression is bounded by the assumed $L_{i-1}^{2p}$-bounds on $(\phi,a)$; Use the sequence of inclusion 
			$$L_{k-1}^2\subset L_{k-2}^2\subset L_{k-3}^{2^2}\subset\cdots\subset L_{k-(k-i+1)}^{2^{k-i}}=L_{i-1}^p,$$ 
			we see that the first summand is also bounded, hence the last expression is bounded and we have obtained the $L_i^p$-bounds on $(\phi,a)$. And therefore we conclude that $(\phi,a)$ is $L_k^2$ bounded.
		\end{proof}
	\end{proposition}

    %The following theorem labelled as 'Theorem $2.6$' in \cite{bauerfuruta2004stable}, together with propostition \ref{boundedness}, implies the first half of theorem \ref{main theorem1}, which is reformulated as corollary \ref{first half of theorem1}. 
    
    The following theorem due to Bauer--Furuta, together with Propostition \ref{boundedness}, implies the first half of Theorem \ref{main theorem1}, which is reformulated as Corollary \ref{first half of theorem1}.

	\begin{theorem}[{\cite[Theorem $2.6$]{bauerfuruta2004stable}}]\label{critical theorem}
		An equivariant Fredholm map $f=l+c:E'\to E$ between $G$-Hilbert space bundles over $Y$ with $E\cong Y\times H$, which extends continuously to the fiberwise one-point completions, defines a stable cohomotopy class 
		$$[f]\in\pi^0_{G,H}(Y;\ind l).$$
		This class is independent of the decomposition of $f$ as a sum. \hfill $\square$\par
	\end{theorem}
	
	\begin{corollary}\label{first half of theorem1}
		The monopole map defines an element $[\mu]$ in the stable cohomotopy group 
		$$\pi^0_{\Z_2,H}(\mathrm{Pic}^c(X);\theta)=\pi^{b^+(X;l)}_{\Z_2,H}(\mathrm{Pic}^c(X);\ind(D)),$$
		where $H$ is a Sobolev completion of $\Gamma(S^-)\oplus\Omega^+(X;\ii\lambda)$. The virtual index bundle $\theta=\ind(D)\ominus H_+$ is the difference of virtual index bundle of the Dirac operator over $\mathrm{Pic}^c(X)$ and the trivial bundle $H_+$ with fiber $H^{2,+}(X;\ii\lambda)$. The $\Z_2$ action on $\ind(D)$ is given by multiplication with $\pm 1$ and on $H_+$ is trivial. \hfill $\square$\par
	\end{corollary}

	As in the Bauer--Furuta theory, there exists a comparison map from the stable equivariant cohomotopy group to $\Z_2$, relating $[\mu]$ with the $\Z_2$ valued $Pin^{-}(2)$ monopole invariant:
	
	\begin{proposition}
		Let $\tilde{X}\to X$ be as above with $b^{+}(X;l)>b_1(X;l)+1$, then we have an homomorphism $t: \pi^{b^+(X;l)}_{\Z_2,H}(\mathrm{Pic}^c(X);\ind(D))\to \Z_2$ which maps the class of the monopole map to the $\Z_2$ valued $Pin^{-}(2)$ monopole invariant.
		\begin{proof}
			any element in $\pi^{b^+(X;l)}_{\Z_2,H}(Pic^c(X);\ind(D))$ is represented by a pointed equivariant map
			$$\mu:TF\to V^+$$
			from the Thom space of a bundle $F$ over $\mathrm{Pic}^c(X)$ to a sphere $V^+=(V'\oplus H^{2,+}(X;\lambda))^+$, where $F-\mathrm{Pic}^c(X)\times V'$ represents the equivariant virtual index bundle of the Dirac operator. The $\Z_2$ valued $Pin^{-}(2)$ monopole invariant is constructed as follows. After possible perturbation of the map $\mu$, the $\Z_2$ fixed point set $TF^{\Z_2}$ is mapped to a subspace of $(V^{\Z_2})^+$ of codimension at least $b^+(X;l)-b_1(X;l)\geq 2$. After possible perturbing further, the preimage of a generic point in the complement is a manifold $M$ with a free $\Z_2$ action. The dimension of $M$ is $ind_{\R}(D)-b^+(X;l)+b_1(X;l)=k$, and the $Pin^{-}(2)$ monopole invariant is the evaluation of the Stiefel-Whitney class of the real vector bundle $(M\times \R^k)/{\Z_2}$ over $M/{\Z_2}$ at the fundamental class.
		\end{proof}
	\end{proposition}
	
	\begin{remark}
    In fact, the assumption $b^{+}(X;l)>b_1(X;l)+1$ can be improved. Following \cite{bauer2004refined}, the monopole map $\mu:\A\to\mathcal{C}$ defines an element in the equivariant stable cohomotopy group $\pi^0_{\Z_2,H}(Q(X,c,V'))$, where the spectrum $Q(X,c,V')$ is defined as $\Sigma^{-V'-H^{2,+}(X;\lambda)}(TF/{\mathrm{Pic}(X)}),$ and there exists a natural homomorphism of this stable cohomotopy group to $\Z_2$.
   		
	\end{remark}
	
	If $b_1(X;l)=0$, the group $\pi^{b^+(X;l)}_{\Z_2,H}(\mathrm{Pic}^c(X);\ind(D))$ simplifies: the index of the Dirac operator is a $d$-dimensional nontrivial $\R$ vector space, where $d=\frac{c_1(c)^2-\sigma(X)}{4}.$
	
	\begin{proposition}\label{isomorphism between equivariant and nonequivariant}
		For $i>1$, the stable equivariant cohomotopy groups $\pi^i_{\Z_2,H}(*;\tilde{\R}^{d})$ are isomorphic to the nonequivariant stable cohomotopy groups $\pi^{i-1}(\R P^{d-1})$ of real projective $(d-1)$-space. In particular, if $X$ is closed $4$-manifold with $b_1(X;l)=0$ and $b^+(X;l)>1$, then the monopole map determines an element in $\pi^{b^+(X;l)-1}(\R P^{d-1})$.
		\begin{proof}
			The long exact stable cohomotopy sequence for the pair$(D(\tilde{\R}^{d}),S(\tilde{\R}^{d}))$ states that
			\begin{align*}
				\cdots\to\tilde{\pi}^{i-1}_{\Z_2,H}(D(\tilde{\R}^{d})^+,S(\tilde{\R}^{d})^+)\to &\tilde{\pi}^{i-1}_{\Z_2,H}(D(\tilde{\R}^{d})^+)\to \tilde{\pi}^{i-1}_{\Z_2,H}(S(\tilde{\R}^{d})^+)\\
				&\to\tilde{\pi}^{i}_{\Z_2,H}(D(\tilde{\R}^{d})^+,S(\tilde{\R}^{d})^+)\to\tilde{\pi}^{i}_{\Z_2,H}(D(\tilde{\R}^{d})^+)\to\cdots,
			\end{align*}
			since the disk $D(\tilde{\R}^{d})$is $\Z_2$-equivariantly contractible, the groups $\tilde{\pi}^i_{\Z_2,H}(D(\tilde{\R}^{d})^+)$ vanish for all $i$, hence $\tilde{\pi}^{i-1}_{\Z_2,H}(S(\tilde{\R}^{d})^+)\cong \tilde{\pi}^{i}_{\Z_2,H}(D(\tilde{\R}^{d})^+,S(\tilde{\R}^{d})^+)$, which is isomorphic to  $\tilde{\pi}^{i}_{\Z_2,H}(D(\tilde{\R}^{d})^+/(S(\tilde{\R}^{d})^+))$.   
			We have $\pi^i_{\Z_2,H}(*;\tilde{\R}^{d}):= \tilde{\pi}^i_{\Z_2,H}((\tilde{\R}^{d})^+)\cong \tilde{\pi}^i_{\Z_2,H}(D(\tilde{\R}^{d})^+/(S(\tilde{\R}^{d})^+))\cong \tilde{\pi}^{i-1}_{\Z_2,H}(S(\tilde{\R}^{d})^+).$ But for the free $\Z_2$-space $S(\tilde{\R}^{d})$ equivariant cohomotopy is isomorphic to the nonequivariant of its quotient. Therefore, $\tilde{\pi}^{i-1}_{\Z_2,H}(S(\tilde{\R}^{d})^+)$ is isomorphic to $\tilde{\pi}^{i-1}((\R P^{d-1})^+)\cong{\pi}^{i-1}(\R P^{d-1}).$
		\end{proof}
	\end{proposition}
	
	Let $k=d-b^+(X;l)$ be the virtual dimension for the moduli space, then we have the following computational results on $\pi^{d-1-k}(\R P^{d-1})$ for small $k$.

	\begin{lemma}\label{main lemma}
	Let $d>1$ be an integer. The Hurewicz map of reduced cohomology groups is defined as following
	\begin{align*}
		h^{d-1-k}:\tilde{\pi}^{d-1-k}(\R P^{d-1}) &\to \tilde{H}^{d-1-k}(\R P^{d-1}),\\
		[f]&\mapsto f^*(1),
	\end{align*}
	with $1\in H^{d-1-k}(S^{d-1-k})\cong \tilde{H}^0(S^0)$ defined by orientation.
    
    %$$d_2:E_2^{p,-1}=H^p(RP^d;\pi^{-1}_{st}(*)=\Z_2)=\Z_2\to E_2^{p+2,-2}=H^{p+2}(RP^d;\pi^{-2}_{st}(*)=\Z_2)=\Z_2$$
    
    \begin{enumerate}
		\item For $k=0$, it is an isomorphism.
		\item For $k=1$, the cokernel is trivial and the kernel is trivial in the case $d=2m+1$ and $m$ even, and  isomrphic to $\Z_2$ otherwise.
		\item For $k=2$, 
		\begin{enumerate}
			\item for $d=2m$, $m$ odd, kernel is trivial or isomorphic to $\Z_2$.
			\item for $d=2m$, $m$ even, the kernel is isomorphic to a group of order $2$ or $4$.
			\item for $d=2m+1$, $m$ odd, the kernel is isomorphic to a group of order $4$.
			\item for $d=2m+1$, $m$ even, the kernel is isomorphic to a group of order $2$ or $4$ and the cokernel is $\Z_2$.
		\end{enumerate}
		%\item For $k=3$, 
		%\begin{enumerate}
		%	\item for $d=2m$, $m$ odd, it has kernel isomorphic to $\Z_{12}$ and cokernel isomorphic to $\Z_2$.
		%	\item for $d=2m$, $m$ even, the kernel is isomorphic to $\Z_2\oplus\Z_2\oplus\Z_{24}$.
		%	\item for $d=2m+1$, $m$ odd, the kernel is isomorphic to $\Z_2$.
			%\item for $d=2m+1$, $m$ even, i dont know whether d_3: E_3^{2m-4,0}\to E_3^{2m-1,-2}$ is isomorphism or $0$ depending on $d$.
		%\end{enumerate}
		%\item For $k=4$,
		%\begin{enumerate}
		%	\item for $d=2m$, $m$ odd, it is trivial as it is the map between trivial groups.
		%	\item for $d=2m$, $m$ even, the kernel is isomorphic to $\Z_2$.
		%	\item for $d=2m+1$, $m$ odd, the cokernel is isomorphic to $\Z_2$.
		%	\item for $d=2m+1$, $m$ even, the kernel contains a subgroup isomorphic to $\Z_2\oplus\Z_2$.
		%\end{enumerate}
	\end{enumerate}
	
	\begin{proof}
	Using the Atiyah-Hirzebruch spectral sequence with $E_2$-term 
	$$H^p(\R P^{d-1};\pi^{q}_{st}(pt))\Rightarrow \pi^{p+q}_{st}(\R P^{d-1}),$$
	$d_2:E_2^{p,0}\to E_2^{p+2,-1}$ detects the Steenrod square $\operatorname{Sq}^2$, and the facts that $\pi^{-1}_{st}(*)\cong\Z_2$ is generated by the Hopf map $\eta:S^3\to S^2$, $\pi^{-2}_{st}(*)\cong\Z_2$ is generated by $\eta^2$, and $\pi^{-3}_{st}(*)\cong\Z_{24}$ is generated by $\nu:S^7\to S^4$ satisfying $\eta^3=12\nu$, we have:

	\begin{enumerate}
		\item For $d=2m, m\geq 1$, The $E_2$ page is 
		
        \begin{sseq}[xlabelstep=1,ylabelstep=1,grid=chess,
        	xlabels={0;1;2;3;{...};2m-6;2m-5;2m-4;2m-3;2m-2;2m-1},ylabels={0;-1;-2;-3;-4},entrysize=12mm]
        	{0...10}{0...4}
        	\ssdrop{\Z}\ssmove{1}{0}\ssdrop{0}\ssmove{1}{0}\ssdrop{\Z_2}\ssmove{1}{0}\ssdrop{0}\ssmove{1}{0}\ssdrop{\cdots}\ssmove{1}{0}\ssdrop{\Z_2}\ssmove{1}{0}\ssdrop{0}\ssmove{1}{0}\ssdrop{\Z_2}\ssmove{1}{0}\ssdrop{0}\ssmove{1}{0}\ssdrop{\Z_2}\ssmove{1}{0}\ssdrop{\Z}
        	\ssmoveto{0}{1}\ssdrop{\Z_2}\ssmove{1}{0}\ssdrop{\Z_2}\ssmove{1}{0}\ssdrop{\Z_2}\ssmove{1}{0}\ssdrop{\Z_2}\ssmove{1}{0}\ssdrop{\cdots}\ssmove{1}{0}\ssdrop{\Z_2}\ssmove{1}{0}\ssdrop{\Z_2}\ssmove{1}{0}\ssdrop{\Z_2}\ssmove{1}{0}\ssdrop{\Z_2}\ssmove{1}{0}\ssdrop{\Z_2}\ssmove{1}{0}\ssdrop{\Z_2}
        	\ssmoveto{0}{2}\ssdrop{\Z_2}\ssmove{1}{0}\ssdrop{\Z_2}\ssmove{1}{0}\ssdrop{\Z_2}\ssmove{1}{0}\ssdrop{\Z_2}\ssmove{1}{0}\ssdrop{\cdots}\ssmove{1}{0}\ssdrop{\Z_2}\ssmove{1}{0}\ssdrop{\Z_2}\ssmove{1}{0}\ssdrop{\Z_2}\ssmove{1}{0}\ssdrop{\Z_2}\ssmove{1}{0}\ssdrop{\Z_2}\ssmove{1}{0}\ssdrop{\Z_2}
        	\ssmoveto{0}{3}\ssdrop{\Z_{24}}\ssmove{1}{0}\ssdrop{\Z_2}\ssmove{1}{0}\ssdrop{\Z_2}\ssmove{1}{0}\ssdrop{\Z_2}\ssmove{1}{0}\ssdrop{\cdots}\ssmove{1}{0}\ssdrop{\Z_2}\ssmove{1}{0}\ssdrop{\Z_2}\ssmove{1}{0}\ssdrop{\Z_2}\ssmove{1}{0}\ssdrop{\Z_2}\ssmove{1}{0}\ssdrop{\Z_2}\ssmove{1}{0}\ssdrop{\Z_{24}}
        	\ssmoveto{0}{4}\ssdrop{0}\ssmove{1}{0}\ssdrop{0}\ssmove{1}{0}\ssdrop{0}\ssmove{1}{0}\ssdrop{0}
        	\ssmove{1}{0}\ssdrop{\cdots}\ssmove{1}{0}\ssdrop{0}\ssmove{1}{0}\ssdrop{0}\ssmove{1}{0}\ssdrop{0}\ssmove{1}{0}\ssdrop{0}\ssmove{1}{0}\ssdrop{0}\ssmove{1}{0}\ssdrop{0}
        \end{sseq}
        \begin{itemize}
        	\item For $m$ an odd number, the $E_3$ page is
        	
        	\begin{sseq}[xlabelstep=1,ylabelstep=1,grid=chess,
        		xlabels={0;1;2;3;{...};2m-6;2m-5;2m-4;2m-3;2m-2;2m-1},ylabels={0;-1;-2},entrysize=12mm]
        		{0...10}{0...2}
                \ssdrop{\Z}\ssmove{1}{0}\ssdrop{0}\ssmove{1}{0}\ssdrop{0}\ssmove{1}{0}\ssdrop{0}\ssmove{1}{0}\ssdrop{\cdots}\ssmove{1}{0}\ssdrop{\Z_2}\ssmove{1}{0}\ssdrop{0}\ssmove{1}{0}\ssdrop{0}\ssmove{1}{0}\ssdrop{0}\ssmove{1}{0}\ssdrop{\Z_2}\ssmove{1}{0}\ssdrop{\Z}
                \ssmoveto{0}{1}\ssdrop{*}\ssmove{1}{0}\ssdrop{*}\ssmove{1}{0}\ssdrop{*}\ssmove{1}{0}\ssdrop{*}\ssmove{1}{0}\ssdrop{\cdots}\ssmove{1}{0}\ssdrop{0}\ssmove{1}{0}\ssdrop{*}\ssmove{1}{0}\ssdrop{*}\ssmove{1}{0}\ssdrop{*}\ssmove{1}{0}\ssdrop{0}\ssmove{1}{0}\ssdrop{\Z_2}
                \ssmoveto{0}{2}\ssdrop{*}\ssmove{1}{0}\ssdrop{*}\ssmove{1}{0}\ssdrop{*}\ssmove{1}{0}\ssdrop{*}\ssmove{1}{0}\ssdrop{\cdots}\ssmove{1}{0}\ssdrop{*}\ssmove{1}{0}\ssdrop{*}\ssmove{1}{0}\ssdrop{*}\ssmove{1}{0}\ssdrop{*}\ssmove{1}{0}\ssdrop{*}\ssmove{1}{0}\ssdrop{*}
        	\end{sseq}
        	
        	\item And for $m$ an even number, the $E_3$ page is
        	
        	\begin{sseq}[xlabelstep=1,ylabelstep=1,grid=chess,
        		xlabels={0;1;2;3;{...};2m-6;2m-5;2m-4;2m-3;2m-2;2m-1},ylabels={0;-1;-2},entrysize=12mm]
        		{0...10}{0...2}
        		\ssdrop{\Z}\ssmove{1}{0}\ssdrop{0}\ssmove{1}{0}\ssdrop{0}\ssmove{1}{0}\ssdrop{0}\ssmove{1}{0}\ssdrop{\cdots}\ssmove{1}{0}\ssdrop{0}\ssmove{1}{0}\ssdrop{0}\ssmove{1}{0}\ssdrop{\Z_2}\ssmove{1}{0}\ssdrop{0}\ssmove{1}{0}\ssdrop{\Z_2}\ssmove{1}{0}\ssdrop{\Z}
        		\ssmoveto{0}{1}\ssdrop{*}\ssmove{1}{0}\ssdrop{*}\ssmove{1}{0}\ssdrop{*}\ssmove{1}{0}\ssdrop{*}\ssmove{1}{0}\ssdrop{\cdots}\ssmove{1}{0}\ssdrop{*}\ssmove{1}{0}\ssdrop{*}\ssmove{1}{0}\ssdrop{0}\ssmove{1}{0}\ssdrop{*}\ssmove{1}{0}\ssdrop{\Z_2}\ssmove{1}{0}\ssdrop{\Z_2}
        		\ssmoveto{0}{2}\ssdrop{*}\ssmove{1}{0}\ssdrop{*}\ssmove{1}{0}\ssdrop{*}\ssmove{1}{0}\ssdrop{*}\ssmove{1}{0}\ssdrop{\cdots}\ssmove{1}{0}\ssdrop{*}\ssmove{1}{0}\ssdrop{*}\ssmove{1}{0}\ssdrop{*}\ssmove{1}{0}\ssdrop{*}\ssmove{1}{0}\ssdrop{\Z_2}\ssmove{1}{0}\ssdrop{*}
        	\end{sseq}
        \end{itemize}
		
		\item For $d=2m+1,m\geq 0$, the $E_2$ page is
	    
	    \begin{sseq}[xlabelstep=1,ylabelstep=1,grid=chess,
	    	xlabels={0;1;2;3;{...};2m-5;2m-4;2m-3;2m-2;2m-1;2m},ylabels={0;-1;-2;-3;-4},entrysize=12mm]
	    	{0...10}{0...4}
	    	\ssdrop{\Z}\ssmove{1}{0}\ssdrop{0}\ssmove{1}{0}\ssdrop{\Z_2}\ssmove{1}{0}\ssdrop{0}\ssmove{1}{0}\ssdrop{\cdots}\ssmove{1}{0}\ssdrop{0}\ssmove{1}{0}\ssdrop{\Z_2}\ssmove{1}{0}\ssdrop{0}\ssmove{1}{0}\ssdrop{\Z_2}\ssmove{1}{0}\ssdrop{0}\ssmove{1}{0}\ssdrop{\Z_2}
	    	\ssmoveto{0}{1}\ssdrop{\Z_2}\ssmove{1}{0}\ssdrop{\Z_2}\ssmove{1}{0}\ssdrop{\Z_2}\ssmove{1}{0}\ssdrop{\Z_2}\ssmove{1}{0}\ssdrop{\cdots}\ssmove{1}{0}\ssdrop{\Z_2}\ssmove{1}{0}\ssdrop{\Z_2}\ssmove{1}{0}\ssdrop{\Z_2}\ssmove{1}{0}\ssdrop{\Z_2}\ssmove{1}{0}\ssdrop{\Z_2}\ssmove{1}{0}\ssdrop{\Z_2}
	    	\ssmoveto{0}{2}\ssdrop{\Z_2}\ssmove{1}{0}\ssdrop{\Z_2}\ssmove{1}{0}\ssdrop{\Z_2}\ssmove{1}{0}\ssdrop{\Z_2}\ssmove{1}{0}\ssdrop{\cdots}\ssmove{1}{0}\ssdrop{\Z_2}\ssmove{1}{0}\ssdrop{\Z_2}\ssmove{1}{0}\ssdrop{\Z_2}\ssmove{1}{0}\ssdrop{\Z_2}\ssmove{1}{0}\ssdrop{\Z_2}\ssmove{1}{0}\ssdrop{\Z_2}
	    	\ssmoveto{0}{3}\ssdrop{\Z_{24}}\ssmove{1}{0}\ssdrop{\Z_2}\ssmove{1}{0}\ssdrop{\Z_2}\ssmove{1}{0}\ssdrop{\Z_2}\ssmove{1}{0}\ssdrop{\cdots}\ssmove{1}{0}\ssdrop{\Z_2}\ssmove{1}{0}\ssdrop{\Z_2}\ssmove{1}{0}\ssdrop{\Z_2}\ssmove{1}{0}\ssdrop{\Z_2}\ssmove{1}{0}\ssdrop{\Z_2}\ssmove{1}{0}\ssdrop{\Z_2}
	    	\ssmoveto{0}{4}\ssdrop{0}\ssmove{1}{0}\ssdrop{0}\ssmove{1}{0}\ssdrop{0}\ssmove{1}{0}\ssdrop{0}
	    	\ssmove{1}{0}\ssdrop{\cdots}\ssmove{1}{0}\ssdrop{0}\ssmove{1}{0}\ssdrop{0}\ssmove{1}{0}\ssdrop{0}\ssmove{1}{0}\ssdrop{0}\ssmove{1}{0}\ssdrop{0}\ssmove{1}{0}\ssdrop{0}
	    \end{sseq}
	    
	    \begin{itemize}
	    	\item For $m$ an odd number, the $E_3$ page is 
	    	
	        \begin{sseq}[xlabelstep=1,ylabelstep=1,grid=chess,
	    		xlabels={0;1;2;3;{...};2m-5;2m-4;2m-3;2m-2;2m-1;2m},ylabels={0;-1;-2},entrysize=12mm]
	    		{0...10}{0...2}
	    		\ssdrop{\Z}\ssmove{1}{0}\ssdrop{0}\ssmove{1}{0}\ssdrop{0}\ssmove{1}{0}\ssdrop{0}\ssmove{1}{0}\ssdrop{\cdots}\ssmove{1}{0}\ssdrop{0}\ssmove{1}{0}\ssdrop{0}\ssmove{1}{0}\ssdrop{0}\ssmove{1}{0}\ssdrop{\Z_2}\ssmove{1}{0}\ssdrop{0}\ssmove{1}{0}\ssdrop{\Z_2}
	    		\ssmoveto{0}{1}\ssdrop{*}\ssmove{1}{0}\ssdrop{*}\ssmove{1}{0}\ssdrop{*}\ssmove{1}{0}\ssdrop{*}\ssmove{1}{0}\ssdrop{\cdots}\ssmove{1}{0}\ssdrop{*}\ssmove{1}{0}\ssdrop{*}\ssmove{1}{0}\ssdrop{*}\ssmove{1}{0}\ssdrop{0}\ssmove{1}{0}\ssdrop{\Z_2}\ssmove{1}{0}\ssdrop{\Z_2}
	    		\ssmoveto{0}{2}\ssdrop{*}\ssmove{1}{0}\ssdrop{*}\ssmove{1}{0}\ssdrop{*}\ssmove{1}{0}\ssdrop{*}\ssmove{1}{0}\ssdrop{\cdots}\ssmove{1}{0}\ssdrop{*}\ssmove{1}{0}\ssdrop{*}\ssmove{1}{0}\ssdrop{*}\ssmove{1}{0}\ssdrop{*}\ssmove{1}{0}\ssdrop{*}\ssmove{1}{0}\ssdrop{\Z_2}
	    	\end{sseq}
	    	
	    	\item And for $m$ an even number, the $E_3$ page is 
	    	
	        \begin{sseq}[xlabelstep=1,ylabelstep=1,grid=chess,
	    		xlabels={0;1;2;3;{...};2m-5;2m-4;2m-3;2m-2;2m-1;2m},ylabels={0;-1;-2},entrysize=12mm]
	    		{0...10}{0...2}
	    		\ssdrop{\Z}\ssmove{1}{0}\ssdrop{0}\ssmove{1}{0}\ssdrop{0}\ssmove{1}{0}\ssdrop{0}\ssmove{1}{0}\ssdrop{\cdots}\ssmove{1}{0}\ssdrop{0}\ssmove{1}{0}\ssdrop{\Z_2}\ssmove{1}{0}\ssdrop{0}\ssmove{1}{0}\ssdrop{0}\ssmove{1}{0}\ssdrop{0}\ssmove{1}{0}\ssdrop{\Z_2}
	    		\ssmoveto{0}{1}\ssdrop{*}\ssmove{1}{0}\ssdrop{*}\ssmove{1}{0}\ssdrop{*}\ssmove{1}{0}\ssdrop{*}\ssmove{1}{0}\ssdrop{\cdots}\ssmove{1}{0}\ssdrop{*}\ssmove{1}{0}\ssdrop{0}\ssmove{1}{0}\ssdrop{*}\ssmove{1}{0}\ssdrop{*}\ssmove{1}{0}\ssdrop{\Z_2}\ssmove{1}{0}\ssdrop{0}
	    		\ssmoveto{0}{2}\ssdrop{*}\ssmove{1}{0}\ssdrop{*}\ssmove{1}{0}\ssdrop{*}\ssmove{1}{0}\ssdrop{*}\ssmove{1}{0}\ssdrop{\cdots}\ssmove{1}{0}\ssdrop{*}\ssmove{1}{0}\ssdrop{*}\ssmove{1}{0}\ssdrop{*}\ssmove{1}{0}\ssdrop{*}\ssmove{1}{0}\ssdrop{*}\ssmove{1}{0}\ssdrop{*}
	    	\end{sseq}
	    \end{itemize}
	\end{enumerate}
	Here ''$*$'' represents $0$ or $\Z_2$ and all conclusions can be read off from these diagrams.
    \end{proof}
    \end{lemma}
	
	%\begin{remark}
		The nontriviality of the kernel of the Hurewicz map suggests that the stable cohomotopy invariant may be a refinement of the $Pin^{-}(2)$-monopole invariant.
	%\end{remark}
	
    \begin{remark}
		Let $\tilde{X}\to X$ be as above with $b_1(X;l)=0$, $b^{+}(X;l)>1$ and suppose $\ind(D)$ is odd, $k=\ind(D)-b^{+}(X;l)=0$, then we have an homomorphism $t: \pi^{b^+(X;l)}_{\Z_2,H}(*;\ind(D))\to \Z$ which is given by the Hurewicz map as in lemma \ref{main lemma}. As stated in \cite{nakamura2015pin}, this maps to the integer valued $Pin^{-}(2)$ monopole invariant. 
	\end{remark}	
	
	\begin{example}\label{classical example Enriques surface}
		Let $(N,c)$ be an Enriques surface equipped with the canonical $Spin^{c-}$ structure $c$ with associated $O(2)$-bundle isomorphic to $\underline{\R}\oplus(l_{K}\otimes\R)$, which arises from a double cover $K\to N$ where $K$ is a $K3$ surface.Then this $Spin^{c-}$ structure $c$ satisfies $b_1(N;l)=0,b^+(N;l)=2$ and $d=ind(D_A)=\frac{\tilde{c}_1(c)^2-\sigma(X)}{4}=2$, which defines an element $[\mu]\in\pi^2_{\Z_2,H}(*;\tilde{\R}^2)=\pi^1(\R P^1)\cong\pi^1(S^1)\cong\pi_0^{st}(S^0)\cong\Z.$ It is computed in \cite{nakamura2020real} that the integer valued $Pin^-(2)$ invariant is equal to $1$ if we specify an orientation of the moduli space. Therefore, $[\mu]$ represents a generator of $\pi^2_{\Z_2,H}(*;\tilde{\R}^2)$, which can be written as $f:\tilde{\R}^2\to \R^2$. By restricting $f$ to the unit spheres, we obtain $\overline{f}:S(\tilde{\R}^2)\to S(R^2)$, which nonequivariantly is a map of degree $2$.
	\end{example}

	\begin{example}[Elliptic surfaces]
		Let $\hat{X}(4m+2)=X(4m+2)/{\iota}$ be the manifold obtained by quotienting the action of an anti-holomorphic free involution $\iota$ on $X(4m+2)$. Then for the canonical $Spin^{c-}$ structure $c_0$ associated to the double cover $X(4m+2)\to \hat{X}(4m+2)$, we have $b_1(\hat{X}(4m+2);l)=0,d=4m+2$ and $k=0$. it is computed \cite{nakamura2020real} that $SW^{Pin}_{\Z}(\hat{X}(4m+2),s)=\pm \binom{2m}{0}=\pm 1$. It turns out that $[\mu(\hat{X}(4m+2))]\in \pi^{4m+2}_{\Z_2,H}(*;\tilde{\R}^{4m+2})=\pi^{4m+1}_{st}(\R P^{4m+1})\cong\Z$ is a generator of this cohomotopy group, which can be represented by a map $q_{4m+1}:\R P^{4m+1}\to S^{4m+1}$ collapsing the subspace $\R P^{4m}$ of $\R P^{4m+1}$ to a point.
	\end{example}

	\section{Connected sum formula}\label{Connected sum formula}
	
	%\subsection{Formulation of the theorem}
		
	\subsection{The setup}
	Let $X=\bigsqcup_{i=1}^n X_i$ be the disjoint union of $n$ closed connected oriented Riemannian $4$-manifolds $X_i$. Suppose each component contains a long neck $N(L)_i=S^3\times[-L,L]$ so that each $X_i=X_i^-\cup X_i^+$ is a union of submanifolds with common boundary $\partial X_i^{\pm}=S^3\times\{0\}$. Also suppose the $Spin^{c-}$-structures are twisted on $X_1^-$, untwisted on other $X_i^{\pm}$s, and $b_1(X_1;l_1)=0$, and $b_1(X_i^{\pm})=0$ for all other $X_i^{\pm}$. The radius of the neck is assumed to be equal in all components and its length $2L$ is assumed to be greater than $8$.
	
	For $\tau\in S_n$ an even permutation, let $X^{\tau}$ be the manifold obtained from $X$ by interchanging the positive parts of its components, 
	$$X_i^{\tau}=X_i^-\cup X_{\tau(i)}^+.$$
	We will compare the stable cohomotopy invariants of $X$ and $X^{\tau}$, and the connected sum formula will be a consequence.
	
	Let $A$ denote a $Spin^{c-}$-connection  which induce the flat connection on the characteristic $O(2)$-bundle $E$ over the long neck. Fix identifications of the spinor bundles and the chosen $Spin^{c-}$-connections over the $n$ copies of $S^3\times[-L,L]$ in $X$.
	
	%Consider the gauge subgroup $\G_r\subset\G$ consisting of gauge transformations which are trivial over the short neck $N(1)=\bigsqcup_{i=1}^n(S^3\times[-1,1])$, i.e. $\G_r=\{g\in\G| g_{|N(1)}\equiv 1\}$. The group $\G_r$ decomposes into a product of gauge groups, each corresponding to one of $X_i^{\pm}$. Let $\ker d_r\subset\ker d$ denote the space of closed $\ii\lambda$-valued $1$-forms on $X$ vanishing identically on the short neck. Using the identification of the connections $A$ and $A^{\tau}$ over the short neck, the space $A+\ker d_r\cong A^{\tau}+\ker d_r$ can be viewed as a subspace of the space of $Spin^{c-}$-connections both over $X$ and $X^{\tau}$. Therefore, $A+\ker d_r/{\G_r}$ is identified with $Pic^c(X)=Pic^c(X^{\tau})=*.$
	
	Let $\A$ and $\mathcal{C}$ denote
	\begin{align*}
		\A&=\Gamma(S^+)\oplus\ker d^*,\\
		\mathcal{C}&=\Gamma(S^-)\oplus\Omega^+(X;\ii\lambda).  
	\end{align*}
	
	The monopole map $\mu:\A\to\mathcal{C}$ is a $\Z_2$-equivariant (restricting to the diagonal subgroup $\Z_2\subset \Z_2 \times ({\times_{i=2}^{n}} S^1)$) map defined by 
	$$(\phi,a) \to (D_{A+a}{\phi},F_{A+a}^+-q(\phi)).$$
	
	Note here we use another construction for the monopole map, different from that of section \ref{The monopole map and the stable cohomotopy invariant}.
	%Let $\A$ and $\mathcal{C}$ denote the quotients
	%\begin{align*}
		%&\A=(A+\ker d_r)\times (\Gamma(S^+)\oplus\Omega^1(X;\ii\lambda)\oplus H^1(X;\ii\lambda))/{\G_r}\\
		%&\mathcal{C}=(A+\ker d_r)\times (\Gamma(S^-)\oplus\Omega^+(X;\ii\lambda)\oplus\Omega^0(X;\ii\lambda))/{\G_r}
	%\end{align*}
	%by the action of the gauge group. Both spaces are just Hilbert space and the monopole map
	
	%$$\mu=\tilde{\mu}/{G_r}:\A\to\mathcal{C}$$
	
	%is a $\Z_2$-equivariant (restricting to the diagonal subgroup $\Z_2\subset \Z_2 \oplus ({\bigoplus_{i=2}^{n}} S^1)$) map defined by
	%$$(A',\phi,a) \to (A',D_{A'+a}{\phi},F_{A'+a}^+-q(\phi),d^*a+f).$$
	The Sobolev completions are $L_k^2$ on $\A$ and $L_{k-1}^2$ on $\mathcal{C}$ for $k\geq 4$.
	
	Consider a smooth path 
	$$\psi:[0,1]\to SO(n)$$
	such that $\psi(0)=\operatorname{Id}, \psi(1)=\tau,$ here $\tau$ is considered as the permutation matrix $(\delta_{i,\tau(j)})_{i,j}\in SO(n)$. And also consider a smooth function $w:S^3\times[-L,L]\to [0,1]$ depending only on the first variable:
	\begin{equation*}
		w(y,t)=\begin{cases}
			0, &\text{if $t\in[-L,-1]$},\\
			1, &\text{if $t\in[1,L]$}.
		\end{cases}
	\end{equation*}
	
	For a section $e$ of a bundle $E$(different from the characteristic $O(2)$-bundle) over $X$, denote by $e_i$ its restriction to the bundle $E_i=E_{|X_i}$. Suppose the restrictions of $E_i$ to the long neck is identified with a bundle $F$. Then $E_{|X_i}$ glue together to form a bundle $E^{\tau}$ over $X^{\tau}$. Any smooth section $e$ patches together to give a smooth section $e^{\tau}$ of $E^{\tau}$: The restrictions $e_i$ to $X_i\setminus N(L)$ remain unchanged. Over the long neck, the restrictions of $e_i$ can be viewed as the components of a section 
	$$\vec{e}=(e_1,\cdots,e_n)$$
	of the bundle $\oplus_{i=1}^n F$ over $S^3\times[-L,L]$. The $i$-th component of the section
	$$\vec{e^{\tau}}=(\psi\circ w)\cdot\vec{e}$$
	now restricts to the section $e_i$ over $[-L,-1]$ and to the section $e_{\tau(i)}$ over $[1,L]$. Patching together, we obtain a smooth section $e^{\tau}$ over $E^{\tau}.$
	
	Applying the gluing construction to forms $\alpha$ and spinors $\phi$ on $X$ defines a linear map $V$
	$$V:(\alpha,\phi)\mapsto(\psi\circ w)\cdot(\alpha,\phi)=(\alpha^{\tau},\phi^{\tau}).$$
	
	Combining all of these constructions, we obtain isomorphisms $\A\to \A^{\tau}$ and $\mathcal{C}\to \mathcal{C}^{\tau}$. Denote all of the isomorphisms by $V$.
	
	\begin{theorem}(The isomorphism theorem)\label{gluing isomorphism}
		Gluing via $V$ induces for $b^+=b_2^+(X;l)=b_2^+(X^{\tau},l^{\tau})$ an isomorphism 
		\begin{align*}
		    \pi^{b^+}_{\Z_2,H}(*;ind(D)) &\to \pi^{b^+}_{\Z_2,H^{\tau}}(*;ind(D^{\tau}))\\
		    [\mu_{X_1}]\wedge (\wedge_{i=2}^n \Res^{S^1}_{Z_2}[\mu_{X_i}]) &\mapsto [\mu_{X_1^{\tau}}]\wedge (\wedge_{i=2}^n \Res^{S^1}_{Z_2}[\mu_{X_i^{\tau}}])
		\end{align*}
		which identifies the class of the monopole maps of $X$ and $X^{\tau}$.
	\end{theorem}
	
	We will give the proof of this theorem in section \ref{proof of the isomorphism theorem}.
	
	\begin{proposition}
		The monopole map $\mu$ on $X=\bigsqcup_{i=1}^n X_i$ is the product of the monopole maps on components of $X$
 		$$\mu={\prod\limits_{i=1}^{n}} \mu_i:\A={\prod\limits_{i=1}^{n}} \A_i\to{\prod\limits_{i=1}^{n}} \mathcal{C}_i=\mathcal{C}.$$
		The associated stable equivariant cohomotopy element is the smash product
		$$[\mu]=\wedge_{i=1}^n[\mu_i]\in\pi^{b^+}_{G_i^n,\oplus H_1}(*;\ind(D)).$$
		$G_i^n$\textnormal{(=$\Z_2$ or $S^1$)} acts on $\oplus H_i$ factorwisely.
	\end{proposition}
	
	\begin{proposition}
		On $S^4$, the restriction map $Res^{S^1}_{Z_2}: \pi^0_{S^1}(*)\to \pi^0_{\Z_2}(*)$, maps the stable cohomotopy element 
		$[\mu]=[id]\in\pi^0_{\Z_2,H}(*)\cong\Z$ associated to the standard $Spin^c$-structure, represented by $f:(\R^m\oplus\C^d)^+\to(\R^m\oplus\C^d)^+$, to the element $
		\bar{f}:(\R^m\oplus\tilde{\R}^{2d})^+\to (\R^m\oplus\tilde{\R}^{2d})^+$ representing $[\bar{f}]\in\pi^0_{\Z_2,H}(*)\cong A(Z_2)\cong \Z[Z_2]\oplus\Z[Z_2/{Z_2}]$, where $[\bar{f}]$ corresponds to $deg(\bar{f}^{Z_2})[Z_2/{Z_2}])$.
		\begin{proof}
			Proposition $2.3$ in \cite{bauer2004stable} gives the proof that the isomorphism $\pi^0_{\Z_2,H}(*)\cong\Z$ is induced by restriction to fixed point sets and $[\mu]$ is the identity. According to \cite{segal1970equivariant}, we have the isomorphisms $\pi^0_{\Z_2,H}(*)\cong A(Z_2)\cong \Z[Z_2]\oplus\Z[Z_2/{Z_2}]$ where $A(\Z_2)$ is the Burnside ring. The restriction map $\Res^{S^1}_{Z_2}$ maps the identity in $\pi^0_{S^1}(*)$ to the identity in $\pi^0_{\Z_2}(*)$ by restricting to the $Z_2$ fixed part.  
		\end{proof}
	\end{proposition}
	
	\subsection{Proof of Theorem \ref{main theorem2}}
	
	\begin{proof}
		Let $n=3$ and let $X=X_1\sqcup X_2\sqcup X_3$, $X_1=X_1\# S^4, X_2=S^4\# X_2, X_3=S^4\# S^4$, and let $\tau(123)=(312)$. Then after applying $\tau$ to the components of $X$, we obtain $X^{\tau}$ which is the disjoint union of $X_1\# X_2$ and two copies of $S^4$. By Theorem \ref{gluing isomorphism}, $[\mu_{X_1}]\wedge \Res^{S^1}_{Z_2}[\mu_{X_2}]=[\mu_{X_1}]\wedge \Res^{S^1}_{Z_2}[\mu_{X_2}]\wedge[id]=[\mu_{X_1\# X_2}]\wedge[id]\wedge[id]=[\mu_{X_1\# X_2}]$.
	\end{proof}

    \subsection{Proof of Theorem \ref{gluing isomorphism}}\label{proof of the isomorphism theorem}
    
    Let $\mu$ and $\mu^{\tau}$ denote the monopole maps on $X$ and $X^{\tau}$ respectively. We will prove the follwing diagram commutes up to suitable homotopy which we will explain later.
 
     \begin{figure}[h]
     	\centering
          \[
          \begin{tikzcd}\label{}
          	\A \arrow[r, "\mu"] \arrow[d, "V"'] & \mathcal{C} \arrow[d, "V"] \\
          	\A^{\tau} \arrow[r, "\mu^{\tau}"] & \mathcal{C}^{\tau}
          \end{tikzcd}
          \]
     	\caption{homotopy commutative diagram}
     	\label{homotopy commutative diagram}
     \end{figure}

	Consider homotopies of Fredholm maps 
	$$\mu_t=l_t+c_t:\A\to \mathcal{C},$$
	starting from $\mu_0=\mu$ and ending at $\mu_1=V^{-1}\mu^{\tau}V$ with the following property: there exists a bounded disk $D\subset\A$ with bounding sphere $S$ such that:
	$$\mu_t^{-1}(0)\subset D\quad \text{for $t\in\{0,1\}$},$$
	and
	$$\mu_t^{-1}(0)\cap S=\emptyset\quad \text{for $t\in[0,1]$}.$$

    \subsubsection{The standard estimates}
    
    All estimates in our proof will be variations of the standard estimates used in the proof of compactness. Consider the point $(\phi,a)$ with $\mu(\phi,a)=0$. Let us recall the three main steps:
    
    Step1. Applying the Weitzenb$\ddot{o}$ck formula to $D_{A'}=D_{A+a'}$ to get a pointwise estimate:
    
    \begin{align*}
    	\Delta(|\phi|^2) &\leq 2\langle D_{A'}^*D_{A'}\phi-\frac{\kappa}{4}\phi-\frac{F_{A'}^+}{2}\phi,\phi\rangle\\
    	&=2\langle -\frac{\kappa}{4}\phi-\frac{1}{4}|\phi|^2\phi,\phi\rangle\\
    	&=\langle-\frac{\kappa}{2}\phi-\frac{1}{2}|\phi|^2\phi,\phi\rangle\\
    	&=-\frac{\kappa}{2}|\phi|^2-\frac{1}{2}|\phi|^4.
    \end{align*}
     
    Evaluate $|\phi(x)|^2$ at a maximum point $x_0$ so that $\Delta(|\phi|^2)\geq 0$, and we obtain an estimate $\kappa(x_0)|\phi(x_0)|^2+|\phi(x_0)|^4\leq 0$, which implies either $\phi\equiv 0$ or $|\phi|^2\leq S=\max\limits_{x\in X}(-\kappa(x),0)$.
    
    Step2. Using the Sobolev embedding $L_1^p\hookrightarrow C^0$ for some $p>4$, we obtain an estimate $\|a\|_{C^0}\leq C_1\|a\|_{L_1^p}$. By elliptic estimate on the operator $d^++d^*$, we have $\|a\|_{L_1^p}\leq C_2(\|d^+a\|_{L^p}+\|d^*a\|_{L^p})$. Combine the two estimates with the equality $d^+a=-F_A^++q(\phi)$ to get 
    $$\|a\|_{C^0}\leq C_1C_2(\|F_A^+\|_{L^p}+\|q(\phi)\|_{L^p}).$$
    
    Step3.(Bootstrapping) Consider the operator $D_A+d^+$ and  assume $L_{i-1}^{2p}$-bounds on $(\phi,a)$ inductively for $i\leq k$ and $p=2^{k-i}$. We have $$\|(\phi,a)\|^p_{L_i^p}-\|(\phi,a)\|^p_{L^p}= \|(D_A\phi,d^+a)\|^p_{L_{i-1}^p}=\|(-a\phi,-F_A^++q(\phi))\|^p_{L_{i-1}^p}.$$
    The last summand is bounded by the assumed $L_{i-1}^{2p}$-bounds on $(\phi,a)$.  
    
    \subsubsection{Varying the length of the long neck}
    
    Let $X'=X \setminus N(L-1)$.
    
    \begin{proposition} \label{phi}
    	There exists a constant $C_3$ independent of $L$ such that for any smooth partition of unity $\psi_+,\psi_-$ on $X$ satisfying 
    	\begin{equation*}
    		\psi_+=\begin{cases}
    			      1 &\text{on $X^+\cap X'$},\\
    			      0 &\text{on $X^-\cap X'$}.
    		       \end{cases}
    	\end{equation*}
    	the following elliptic estimate holds:
    	$$\|a\|_{C^0(X')}\leq C_3(\|(d^*+d^+)\psi_{+}a\|_{L^p}+\|(d^*+d^+)\psi_{-}a\|_{L^p}).$$
    	\begin{proof}
    		%We first show $C_1$ is independent of $L$. Choose for any $x\in X$ a bump function $\beta_x:X\to [0,1]$ supported near $x$ and $\beta_x(x)=1$. $\beta_x$ can be chosen such that $||\beta_x||_{C^1(X)}\leq M$ where $M$ is independent of $L$. Let $c$ be any Sobolev constant for a particular $L$. Then the inequality
    		%\begin{equation}
    			%||a||_{C^0(X)}=\max\limits_{x\in X}|\beta_xa|\leq c\max\limits_{x\in X}||\beta_xa||_{L_1^p}\leq Mc||a||_{L_1^p}. \label{Sobolev embedding estimate}
    		%\end{equation}
    		%shows $C_1=Mc$ is independent of $L$.
    		
    		%Now we consider the constants in the elliptic estimate. 
    		We will work with manifolds with tubular ends: Add $n$ tubes $S^3\times [0,+\infty)$ to $X^-$, and $n$ tubes $S^3\times (-\infty,0]$ to $X^+$ respectively to obtain $Y^{\pm}$(compare \cite{nakamura2015pin} section 5). 
    		
    		Let $f^-_{\alpha}\geq 0$ be a smooth function on $Y^-$ such that 
    		\begin{equation*}
    			f^-_{\alpha}(y)=
    			        \begin{cases}
    				        0 &\text{if $y\notin\bigsqcup_{i=1}^n S^3\times(-L+2,\infty)$,}\\
    				        \alpha(r+L-2) &\text{if $y=(r,s)\in S^3\times[-L+3,\infty)$.}   
    			        \end{cases}
    		\end{equation*}
    		And similarly for $f^+_{\alpha}\geq 0$ on $Y^+$:
    		\begin{equation*}
    			f^+_{\alpha}(y)=
    			\begin{cases}
    				0 &\text{if $y\notin\bigsqcup_{i=1}^n S^3\times(-\infty,L-2)$,}\\
    				-\alpha(r-L+2) &\text{if $y=(r,s)\in S^3\times(-\infty,L-3]$.}   
    			\end{cases}
    		\end{equation*}
    		Define the weighted Sobolev spaces $L_k^{p,\alpha}(Y^{\pm})$ on sections of bundles over $Y^{\pm}$ using the norm 
    		$$\|h\|_{L_k^{p,\alpha}(Y^{\pm})}:=\|\exp(f_{\alpha}^{\pm})h\|_{L_k^p(Y^{\pm})}.$$
    		
    		Denote the pull-back bundle $\pi^*(\Lambda^k T^*Z)\otimes\lambda\to Z\times\R$ by $\Lambda^k_Z$, for $Z$ a closed oriented $3$-manifold and $\pi:Z\times\R\to Z$ the natural projection to the first factor. By the identification
    		$$\Omega^1(S^3\times\R)\cong \Gamma(\Lambda_{S^3}^0\oplus \Lambda_{S^3}^1)\cong \Omega^0(S^3\times \R)\oplus \Omega^+(S^3\times \R), $$
    		we can regard $d^*+d^+$ as $\frac{\partial}{\partial t}+L$, where $L= \begin{pmatrix}
    			0 & d^* \\
    			d & *d
    		\end{pmatrix}$
    		 is a self adjoint elliptic operator on $S^3$, satisfying $L^2=\Delta$. According to \cite{donaldson2002floer}, after appropriate completion, the operator $\frac{\partial}{\partial t}+L$ induces a Fredholm map $\frac{\partial}{\partial t}+L:L_1^{p,\alpha}((\Omega^0_{Y^{\pm}}\oplus\Omega^{2,+}_{Y^{\pm}})(\lambda))\to L^{p,\alpha}((\Omega^1_{Y^{\pm}})(\lambda))$ if $\alpha\notin Spec(L).$ Suppose $\alpha<0$ and $\alpha$ is greater than the maximal negative eigenvalue of $L$, using the fact \cite{atiyah1975spectral} that the $L^2$ harmonic forms $\mathcal{H}(Y^{\pm})$ is isomorphic to the image $\hat{H}(X^{\pm})$ of $H^*_{cpt}(Y^{\pm})\to H^*(Y^{\pm})$(or equivalently $H^*(Y^{\pm},S^3)\to H^*(Y^{\pm})$), then $\ker(\frac{\partial}{\partial t}+L)\cong H^1(Y^{\pm};\lambda^{\pm})$, $coker (\frac{\partial}{\partial t}+L)\cong H^0(Y^{\pm};\lambda^{\pm})\oplus H^+(Y^{\pm};\lambda^{\pm})$. Choose a $1$-form $a\in\Omega^1(X;\lambda)$, then we can form compactly supported forms $b^{\pm}=\psi_{\pm}a$ on $Y^{\pm}$ and we have the following elliptic estimate \cite{lockhart1985elliptic}
    		 $$\|b^{\pm}\|_{L_1^{p,\alpha}(Y^{\pm})}\leq C^{\pm}(\|(d^*+d^+)b^{\pm}\|_{L^{p,\alpha}(Y^{\pm})}+\|pr(b)\|).$$ Since by assumption $H^1(Y^{\pm};\lambda)=0$, the estimate simplifies to 
    		 $$\|b^{\pm}\|_{L_1^{p,\alpha}(Y^{\pm})}\leq C^{\pm}\|(d^*+d^+)b^{\pm}\|_{L^{p,\alpha}(Y^{\pm})}.$$
    		 
    		 Choose for any $x\in X$ a bump function $\beta_x:X\to [0,1]$ supported near $x$ and $\beta_x(x)=1$. $\beta_x$ can be chosen such that $\|\beta_x\|_{C^1(X)}\leq M$ where $M$ is independent of $L$. Then $\|a\|_{C^0(X)}=\max\limits_{x\in X}|\beta_x a|$.
    		 
    		 \begin{claim}
    		 	For $Y_i^{+}=X_i^{+}\cup([-\infty,0]$, let $\theta\in\Omega^1(Y_i^+;\lambda)$ with $supp(\theta)\subseteq Y_i^+\setminus (-\infty,1-L)$, and for simplicity we denote $Y_i^+$ by $Y^+$. Then there exists a constant $C_+$ which is independent of $L$ such that $\|\theta\|_{C^0(Y^{+})}\leq C_+\|\theta\|_{L^p_1(Y^+)}$.
    		 	\begin{proof}
    		 		For $x\in Y^+\setminus (-\infty,1-L)$, there exists a constant $c$ which is independent of $L$ such that $\|\beta_x\theta\|_{C^0(Y^{+})}\leq c\|\beta_x\theta\|_{L^p_1(Y^+)}$, which can be seen as follows:
    		 		\begin{enumerate}
    		 			\item If $x\in S^3\times [1-L,-1]$, the support of $\beta_x\theta$ is contained in $S^3\times[1-L,0]$. Suppose for convenience that $L$ is an integer, and since the metric on the neck is the product metric, we have by the Sobolev embedding $$\|\beta_x\theta\|_{C^0(S^3\times[-k-1,-k])}\leq C_s\|\beta_x\theta\|_{L^p_1(S^3\times[-k-1,-k])}
    		 		    \ \text{for any}\ 0\leq k\leq L-2,$$
    		 		    Combine all estimates on the neck to get  
    		 			\begin{align*}
    		 				\|\beta_x\theta\|_{C^0(S^3\times[1-L,0])}&\leq \sum_{k=0}^{L-2} \|\beta_x\theta\|_{C^0(S^3\times[-k-1,-k])}\\
    		 				&\leq \sum_{k=0}^{L-2} C_s\|\beta_x\theta\|_{L^p_1(S^3\times[-k-1,-k])}\\
    		 				&=C_s\|\beta_x\theta\|_{L^p_1(S^3\times[1-L,0])}.
    		 			\end{align*}
    		 			\item If $x\in Y^+\setminus (-\infty,-1)$, since $Y^+\setminus (-\infty,-1)$ is compact, there exists a constant $C_t$ such that $\|\beta_x\theta\|_{C^0(Y^+\setminus (-\infty,-1))}\leq C_t\|\beta_x\theta\|_{L^p_1(Y^+\setminus (-\infty,-1))}$.
      		 		\end{enumerate}
      		 		Take $c=max(C_s,C_t)$ and we have
      		 		$$\|\theta\|_{C^0(Y^+)}=\|\theta\|_{C^0(Y^+\setminus(-\infty,1-L))}=\max\limits_{x\in (Y^+\setminus(-\infty,1-L))}|\beta_x\theta|\leq \max\limits_{x\in (Y^+\setminus(-\infty,1-L))}c\|\beta_x\theta\|_{L^p_1(Y^+)}\leq cM\|\theta\|_{L^p_1(Y^+)}.$$
      		 		Take $C_+=cM$ as desired.
    		 	\end{proof}	 	
    		 \end{claim}
    		 
    		 We have similar results for $Y_i^-$ and we denote $\max(C_+,C_-)$ by $C$.	
    		 
    		 Now we are ready to prove the estimate:
    		 \begin{align*}
    		 	\|a\|_{C^0(X')}&\leq \|\exp(f_{\alpha}^+)b^+\|_{C^0(Y^+)}+\|\exp(f_{\alpha}^-)b^-\|_{C^0(Y^-)}\\
    		 	&\leq C(\|\exp(f_{\alpha}^+)b^+\|_{L^p_1(Y^+)}+\|\exp(f_{\alpha}^-)b^-\|_{L^p_1(Y^-)})\\
    		 	&=C(\|b^+\|_{L^{p,\alpha}_1(Y^+)}+\|b^-\|_{L^{p,\alpha}_1(Y^-)})\\
    		 	&\leq C(C^{+}\|(d^*+d^+)b^{+}\|_{L^{p,\alpha}(Y^{+})}+C^{-}\|(d^*+d^+)b^{-}\|_{L^{p,\alpha}(Y^{-})})\\
    		 	&\leq C(C^{+}\|(d^*+d^+)b^{+}\|_{L^{p}(Y^{+})}+C^{-}\|(d^*+d^+)b^{-}\|_{L^{p}(Y^{-})})\\
    		 	&\leq C(C^{+}\|(d^*+d^+)b^{+}\|_{L^{p}(X\setminus(X'\cap X^-))}+C^{-}\|(d^*+d^+)b^{-}|\|_{L^{p}(X\setminus(X'\cap X^+))})\\
    		 	&\leq C(C^{+}\|(d^*+d^+)b^{+}\|_{L^{p}(X)}+C^{-}\|(d^*+d^+)b^{-}\|_{L^{p}(X)}).
    		 \end{align*}
    		 Take $C_3=\max(CC^+,CC^-)$ to obtain the desired estimate.
    		 %using the fact that $|a(x)|=|b^{\pm}(x)|$ for any $x\in X'$, combined with Sobolev inequality \ref{Sobolev embedding estimate}, we obtain the desired estimate with constant $C_3=C_1\max(C^{\pm})$.     		
    	\end{proof}	
    \end{proposition}
    
    \subsubsection{The first homotopy}
    
    For $R\leq L$, let $\rho_{R}:X\to [0,1]$ and $\rho_{R}^{\tau}\to [0,1]$ be smooth cutoff functions such that: 
    \begin{itemize}
    	\item $\rho_{R}=\rho_{R}^{\tau}$ on the components $X_i^{\pm}$,
    	\item $\rho_{R}=\rho_{R}^{\tau}=0$ On $N(R-1)$,
    	\item $\rho_{R}=\rho_{R}^{\tau}=1$ for $x\notin N(R)$,
    	\item $\rho_{R}$ and $\rho_{R}^{\tau}$ are constants on $S^3\times\{r\}$ on $N(R)\setminus N(R-1)$. 
    \end{itemize}
    The homotopy 
    $$\rho_{R,t}:=(1-t)+t\rho_R, t\in[0,1]$$
    describes a homotopy from constant map $1$ to $\rho_R$ on X. Similarly, we can define the homotopy $\rho_{R,t}^{\tau}$.
    
    Consider the homotopy $\mu_t:\A \to \mathcal{C}$ defined by
    $$\mu_t(\phi,a)=(D_{A+a}\phi,F_{A+a}^+-\rho_{L,t}q(\phi)).$$
    \begin{lemma}
    	$\mu_t^{-1}(0)$ is uniformly bounded for all $t\in [0,1]$.
    	\begin{proof}
    		At the maximum point $x_0$ of $|\phi|^2$, by the estimate in step1, we have the estimate
    		$$\kappa|\phi|^2+\rho_{L,t}|\phi|^4\leq 0.$$
    		Since $\kappa>0$ over the neck, the maximum point $x_0$ is not on the long neck and the maximum satisfies the estimate $|\phi(x_0)|^2\leq S=\max\limits_{x\in X}(-\kappa,0)$. The norm of $\rho_{L,t}q(\phi)$ is bounded by a multiple of the norm of $q(\phi)$. These bounds are independent of $t$.
    	\end{proof}
    \end{lemma}
    
    \subsubsection{The second homotopy}
    The homotopy $$\mu_{t+1}(\phi,a)=(D_{A+\rho_{2,t}a}\phi,F_{A+a}^+-\rho_L q(\phi))$$
    starts at $\mu_1$ and ends at $P$, where $P$ is linear on the short neck $N(1)$.

    \begin{lemma}\label{lemma1 of second homotopy}
    	There exist constants $U$ and $L_0$ such that for any $(\phi,a)$ which is a zero of $\mu_1$ on $X(L)$ with $L\geq L_0$,we have the estimate $\|a\|_{C^0}\leq U$.
    	\begin{proof}
    		On $N(L-1)$, $0=F^+_{A+a}-\rho_Lq(\phi)=F^+_{A+a}$ implies $d^+a=0$ since we have assumed that $A$ is flat over the long neck. Combined with $a\in\ker d^*$, we can deduce that $a$ is harmonic on the long neck. Because $Ric(x)\geq 0$ for $x\in N(L-1)$, we have the maximum principle for $a$ over $N(L-1)$:
    	    $$\|a\|_{C^0(N(L-1))}\leq \|a\|_{C^0(X')}=\|a\|_{C^0(X)}.$$
    	      	   
    	    Let $\psi_{\pm}$ be the partition of unity as in proposition \ref{phi}, with $|d\psi_{\pm}|<L^{-1}$ over the long neck. From Proposition \ref{phi} we have the estimate
    		$$\|a\|_{C^0(X')}\leq C_3(\|(d^*+d^+)\psi_+a\|_{L^p}+\|(d^*+d^+)\psi_-a\|_{L^p}).$$
    		
    		We compute 
    		$$\|(d^*+d^+)\psi_+a\|_{L^p(X)}\leq \|\psi_{\pm}(d^*+d^+)a\|_{L^p(X)}+\|d\psi_{\pm}\|_{L^p(X)}\|a\|_{C^0(N(L-1))}.$$
    		
    		Combined with $(d^*+d^+)a=0$ on $N(L-1)$ and the maximum principle, we obtain
            \begin{align*}
            	\|a\|_{C^0(X)} &\leq C_3(\|(d^*+d^+)\psi_+a\|_{L^p(X)}+\|(d^*+d^+)\psi_-a\|_{L^p(X)})\\
            	&\leq C_3(\|\psi_{+}(d^*+d^+)a\|_{L^p(X)}+\|d\psi_{+}\|_{L^p(X)}\|a\|_{C^0(N(L-1))}\\
            	&+\|\psi_{-}(d^*+d^+)a\|_{L^p(X)}+\|d\psi_{-}\|_{L^p}\|a\|_{C^0(N(L-1))})\\
            	&\leq C_3(\|\psi_{+}\|_{C^0(X')}\|(d^*+d^+)a\|_{L^p(X^-\cap X')}+\|\psi_{-}\|_{C^0(X')}\|(d^*+d^+)a\|_{L^p(X^-\cap X')})\\
            	&+C_3(\|d\psi_+\|_{L^p(N(L-1))}+\|d\psi_-\|_{L^p(N(L-1))})\|a\|_{C^0(N(L-1))}\\
            	&\leq C_3\|(d^*+d^+)a\|_{L^p(X')}+C_3 \cdot 2(2(L-1)\operatorname{Vol}(S^3))^{\frac{1}{p}}\cdot L^{-1}\|a\|_{C^0(N(L-1))}\\
            	&\leq C_3\|q(\psi)-F_A^+\|_{C^0(X)}\operatorname{Vol}(X')^{\frac{1}{p}}+3C_3 L^{\frac{1}{p}-1}\operatorname{Vol}(S^3)^{\frac{1}{p}}\|a\|_{C^0(X)}\\
            	&\leq C_3(\|q(\psi)\|_{C^0(x)}+\|F_A^+\|_{C^0(X)})\operatorname{Vol}(X')^{\frac{1}{p}}+3C_3 L^{\frac{1}{p}-1}\operatorname{Vol}(S^3)^{\frac{1}{p}}\|a\|_{C^0(X)}\\
            	&\leq C_3(\frac{1}{2}S+\|F_A^+\|_{C^0})Vol(X')^{\frac{1}{p}}+3C_3 L^{\frac{1}{p}-1}\operatorname{Vol}(S^3)^{\frac{1}{p}}\|a\|_{C^0(X)}.
            \end{align*}
    		Take $U=2C_3(\frac{1}{2}S+\|F_A^+\|_{C^0})\operatorname{Vol}(X')^{\frac{1}{p}}$, and $L_0=(6C_3)^{\frac{p}{p-1}}\operatorname{Vol}(S^3)^{\frac{1}{p-1}}$ 
    		then the claim follows. 
    		%When varying the background connection $A$ in $A+\ker d_r$, the curvature $F_A^+$ remains unchanged. Therefore, we can choose $U$ on all fibers over $Pic^c(X)$ consistenly.
    	\end{proof}
    \end{lemma}
    
    \begin{lemma} \label{phi leq 2S, so alsp leq S}
    	If the length $L\geq L_1\geq L_0$, then the following holds: 
    	For any $(\phi,a)\in\mu_t^{-1}(0)$ for $1\leq t\leq 2$, if it satisfies $|\|\phi\|_{C^0}^2\leq 2S$ and $\|a\|_{C^0}\leq 2U$, then it satisfies the stricter estimates $\|\phi\|_{C^0}^2\leq S$ and $\|a\|_{C^0}\leq U$.
    	\begin{proof}
    		On $N(L-1)$, $a$ is harmonic and decomposes as $a=a_i+a_s$, where $a_i$ points in the $\R$ direction and $a_s$ points in the $S^3$ direction.
    		
    		For $a_s$, we have by the Bochner's formula
    		\begin{align*}
    			\Delta_X|a_s|^2&=-2\langle\Delta_{X,H}a_s,a_s\rangle+2|\nabla a_s|^2+2Ric_X(a_s,a_s)\\
    			&=2|\nabla a_s|^2+2Ric_X(a_s,a_s)\\
    			&\geq 2Ric_X(a_s,a_s)\\
    			&\geq \delta^2|a_s|^2,
    		\end{align*}
    		where the last inequality follows from the fact that the Ricci tensor is positive and bounded below by some constant $\delta^2, \delta>0$ in the direction of $S^3$.
    		
    		For $t\in [1-L,L-1]$, let $\alpha=\sum_{i=1}^{n}|a_{s,i}|^2, A(t)=\max\limits_{y\in S^3}\alpha(y,t)$ and choose for every $t$ a corresponding $y_t$ such that $\alpha(y_t,t)=A(t)$.The quantity $\alpha$ satisifes the following differential inequalities
    		$$\frac{d^2\alpha}{dt^2}=\Delta_X\alpha-\Delta_{S^3}\alpha\geq \delta^2\alpha-\Delta_{S^3}\alpha\geq \delta^2\alpha+\sigma^2\alpha=(\delta^2+\sigma^2)\alpha,$$
    		where the last inequality holds because $\Delta_{S^3}$ has spectral gap: $\Delta_{S^3} f\leq -\sigma^2 f$ for some $\sigma>0$ for all $f\in (\ker\Delta_{S^3})^{\perp}.$
    		
    		Apply this differential inequality to $A(t)$, we obtain
    		$$A''(t)\geq (\delta^2+\sigma^2)A(t).$$
    		
    		Let $z(t)$ solve the ODE: $z''(t)=(\delta^2+\sigma^2)z$ with boundary values $z(\pm(L-1))=4nU^2$, then the symmetric solution about $t=0$ is 
    		$$z(t)=\frac{4nU^2\cosh(\sqrt{\delta^2+\sigma^2}t)}{\cosh(\sqrt{\delta^2+\sigma^2}(L-1))}.$$
    		
    		Consider now the function $w(t):=A(t)-z(t)$, then it satisfies the differential inequality
    		$$w''-(\delta^2+\sigma^2)w=(A''-(\delta^2+\sigma^2)A)-(z''-(\delta^2+\sigma^2)z)\geq 0,$$
    		and has boundary condition $w(\pm(L-1))\leq 0$ since by assumption $||a||_{C^0}\leq 2U$.
    		
    		The maximum principle for the operator $L[w]:=w''-(\delta^2+\sigma^2)w$ implies $w\leq 0$ for $t\in[1-L,L-1]$, hence 
    		$$A(t)\leq z(t)=\frac{4nU^2\cosh(\sqrt{\delta^2+\sigma^2}t)}{\cosh(\sqrt{\delta^2+\sigma^2}(L-1))}.$$ 
    		In particular, $|a_s|^2\leq \frac{4nU^2\cosh(\sqrt{\delta^2+\sigma^2}t)}{\cosh(\sqrt{\delta^2+\sigma^2}(L-1))}$ and there is exponential decay of $|a_s|^2$ towards the middle of $N(L-1).$
    		
    		Now let us give the estimate for the spinor part $|\phi|^2$ and let $A'=A+\rho_{2,t}a$. If $|\phi|^2$ attains maximum inside $N(L-1)$, say, at the point $x_0\in N(L-1)$, we have by the Weitzenb$\ddot{o}$ck formula 
    		\begin{align*}
    			0\leq \Delta|\phi|^2(x_0)&\leq 2\langle D_{A'}^*D_{A'}\phi-\frac{\kappa}{4}\phi-\frac{1}{2}F_{A'}^+\phi,\phi\rangle\\
    			&=-\frac{\kappa}{2}|\phi|^2-\langle F_{A'}^+\phi,\phi\rangle\\
    			&=-\frac{\kappa}{2}|\phi|^2-\langle d^+(\rho_{2,t}a)\phi,\phi\rangle\\
    			&=-\frac{\kappa}{2}|\phi|^2-\langle (d\rho_{2,t}\wedge a)^+\phi+(\rho_{2,t}da)^+\phi,\phi\rangle\\
    			&=-\frac{\kappa}{2}|\phi|^2-\langle (d\rho_{2,t}\wedge a)^+\phi,\phi\rangle\\
    			&=-\frac{\kappa}{2}|\phi|^2-\langle (d\rho_{2,t}\wedge a_s)^+\phi,\phi\rangle,
    		\end{align*}
    		where the last equality follows from $d\rho_{2,t}\wedge a=d\rho_{2,t}\wedge a_s$ since $\rho_{2,t}$ is constant in the direction of $S^3.$ Outside $N(2)$, $d\rho_{2,t}=0$, so we only need to consider the case in which $x_0\in N(2)$. But by the exponential decay of $|a_s|$ with respect to $L$, if we stretch the length of the neck, the second summand in the last inequality will tend to zero and the scalar curvature term will prevail. Therefore, $|\phi|^2$ cannot attain its maximum in $N(L-1)$. 
    		
    		Outside $N(L-1)$, we have $\rho_{2,t}=1$, so $D_{A+a}\phi=0$ and $F_{A+a}^+=\rho_Lq(\phi).$ For $A'=A+a$, we have by the Weitzenb$\ddot{a}$ck formula 
    		\begin{align*}
    			\Delta|\phi|^2&\leq 2\langle D_{A'}^*D_{A'}\phi-\frac{\kappa}{4}\phi-\frac{1}{2}F_{A'}^+\phi,\phi\rangle\\
    			&=-\frac{\kappa}{2}|\phi|^2-\frac{1}{2}\rho_L|\phi|^4.
    		\end{align*}  
    		The laplacian of $|\phi|^2$ is nonnegative at the maximum point $x_0$, so by the last inequality, $\kappa|\phi|^2+\rho_L|\phi|^4\leq 0$, and if $|\phi|(x_0)\neq 0$, we have $\kappa+\rho_L|\phi|^2\leq 0$.
    		\begin{enumerate}
    			\item If $x_0\in N(L)\setminus N(L-1)$, the scalar curvature $\kappa$ is positive, which is a contradiction.
    			\item If $x_0$ is outside $N(L)$, we would have $\rho_L=1$. In this case, we obtain $|\phi|^2\leq S=\max\limits_{x\in X}(0,-\kappa)$, as desired.
    		\end{enumerate} 		
    		
    		Now for the estimate on the norm of the form part $|a|$, since $a$ is harmonic on $N(L-1)$, we have $\|a|\|_{C^0(X)}=\|a\|_{C^0(X')}$ by the maximum principle, which implies $|a|$ would not attain its maximum inside $N(L-1)$. Outside $N(L-1)$, $a$ satisfies the same equations as in lemma \ref{lemma1 of second homotopy}, and we can use the same argument to obtain the desired estimate. 
    		%Note $a$ is harmonic on $N(L-1)$ and it decomposes as $a=a_i+a_s$. For $a_s$ in the direction of $S^3$, we have the following inequality on Laplacian
    		%$$\Delta|a_s|^2\leq -2<Ric(a_s),a_s>.$$
    		%Since $Ric$ is positive definite on $S^3$,
    		%$$-2<Ric(a_s),a_s>\leq -\delta^2|a_s|^2$$
    		%for some $\delta>0$. Let $\alpha$ denote the sum $\sum_{i=1}^{n}|a_{s,i}|^2$, then $\alpha$ satisfies
    		%$$\frac{d^2\alpha}{dr^2}\geq -\Delta_X\alpha\geq \delta^2\alpha,$$
    		%which implies $|a_s|^2\leq \frac{4nU^2\cosh(\delta r)}{\cosh(\delta(L-1))}$ for $1-L\leq r\leq N-1.$ Moreover, $|a_s|$ decays exponentially towards the middle of the neck.
    		
    		%For the estimate of the spinor, if $|\phi|^2$ attains the maximum on some point $x_0\notin N(L-1)$, we can apply step1. Otherwise if the maximum is obtained at a point in $N(L-1)$,it will satisfy
    		%$$0\leq \Delta |\phi|^2\leq -\frac{\kappa}{2}|\phi|^2+<(d\rho_{2,t}\wedge a)^+\phi,\phi>=-\frac{\kappa}{2}|\phi|^2.$$
    		%Since  $\rho_{2,t}$ is constant on $S^3$ direction, $d\rho_{2,t}\wedge a=d\rho_{2,t}\wedge a_s$. Therefore the norm of $<(d\rho_{2,t}\wedge a)^+\phi,\phi>$ decays exponentially with respect to $L$. If we make $L$ so large then the inequality will not hold.
    		%But since $\kappa>0$ on $N(L-1)$, it will be a contradiction.
    		
    		%We can use the same argument as in Lemma 4.2 to get the sharper bound on $||a||$: the bound on $|\phi|^2$ is unchanged since $a$ is harmonic along $N(L-1)$.
    	\end{proof}
    \end{lemma}

    \subsubsection{The third homotopy}
    
    Lastly, we will construct a homotopy between $P$ and $V^{-1}P^{\tau}V$.
    Outside $N(1)$, $P=V^{-1}P^{\tau}V$ since $w$ is constant outside $N(1)$. On $N(1)$, $P(\phi,a)=(D_A\phi,F_{A+a}^+)=(D_A\phi,d^+a)$ so $P$ is a linear differential operator. Both the operators $D_A$ and $d^+$ take the form $\frac{d}{dt}+L$ for a self adjoint operator $L$ on $S^3$ and $V$ commutes with $L$ since $V$ depends only on $t$, and $L$ depends only on the $S^3$ parameters. For $u$ either a spinor or a $1$-form, we have 
    \begin{align*}
    	V^{-1}P^{\tau}V(u)&=V^{-1}P^{\tau}(V(u))\\
    	&=V^{-1}\left(\frac{d}{dt}+L^{\tau}\right)(\psi\circ w(t)(u))\\
    	&=V^{-1}\left(\frac{d}{dt}V\right)(u)+V^{-1}V\left(\frac{d}{dt}u\right)+V^{-1}L^{\tau}V(u)\\
    	&=V^{-1}\left(\frac{d}{dt}V\right)(u)+\frac{du}{dt}+L(u)\\
    	&=V^{-1}\left(\frac{d}{dt}V\right)(u)+P(u)\\
    	&=d\log(V)(u)+P(u).
    \end{align*}
    Therefore, the difference on the short neck is a multiplication operator 
    $$V^{-1}P^{\tau}V=P+d\log (V).$$
    For $t\in[0,1],$ consider the matrix-valued function $\psi\circ tw:S^3\times[-L,L]\to SO(n)$. It induces a map $V_t$ over $N(L)$ which is multiplication : $(\phi,a)\mapsto \psi\circ tw(\phi,a)$. This multiplication is not well defind outside $N(L)$, but the conjugation $V_t^{-1}PV_t=P+d\log(V_t)$ does extend over the whole manifold $X$, for $0\leq t\leq 1$, i.e., $P+d\log(V_t):\A\to \mathcal{C}$. This family of operators consist of the final homotopy.
    
    \begin{lemma}
    	If $L\geq L_2\geq L_1$, then the following holds: 
    	For any $(\phi,a)\in (P+d\log(V_t))^{-1}(0)$ for $0\leq t\leq 1$, if it satisfies $|\phi|_{C^0}^2\leq 2S$ and $\|a\|_{C^0}\leq 2U$, then it satisfies the stronger estimates $\|\phi\|_{C^0}^2\leq S$ and $\|a\|_{C^0}\leq U$.
    	\begin{proof}
    		Let $(\phi,a)$ be a solution to $P+d\log(V_t)=0$:
    		$$(P+d\log(V_t))(\phi,a)=V_t^{-1}P^{\tau}V_t(\phi,a)=0.$$
    		Outside $N(1)$, $(\phi,a)$ is also a solution to $P$. Over $N(L)$, on the other hand, $V_t(\phi,a)$ is a solution to $P^{\tau}$.
    		
    		Let us consider the bound of $|\phi|^2$ outside $N(1)$ and the bound of $|V_t\phi|^2$ on $N(L-1)$. First, on $N(L-1)$, $a=a_i+a_s$ splits into a sum of harmonic parts as in lemma \ref{phi leq 2S, so alsp leq S}. Since $|a_i|^2+|a_s|^2=|a|^2=|V_ta|^2=|(V_ta)_i|^2+|(V_ta)_s|^2$ and $V_t$ depends only on $t$, we have $|a_s|^2=|(V_ta)_s|^2$. On $N(L-1)$, since $V_ta$ solutes $P^{\tau}=0$, it is harmonic. Then exponential decay of $|(V_ta)_s|$ implies the exponential decay of $|a_s|$. 
    		\begin{enumerate}
    			\item If $|\phi|^2$ attains its maximum inside $N(L-1)$, say, at $x_0\in N(L-1)$, and from $|V_t\phi|^2=|\phi|^2$, we have 
    			$$0\leq \Delta|\phi|^2(x_0)\leq\Delta|V_t\phi|^2(x_0)\leq -\frac{\kappa}{2}|V_t\phi|^2-\langle (d\rho^{\tau}_{2}\wedge (V_ta))^+V_t\phi,V_t\phi\rangle.$$
    			Outside $N(2)$, the second summand vanishes and the scalar curvature is positive, so the maximum will not be obtained outside $N(2)$. Inside $N(2)$, If we stretch the neck, exponential decay of $|a_s|$ with respect to $L$ will also imply the right hand side of the inequality is negative and hence the maximum is not obtained inside $N(L-1)$.
    			\item If the maximum is obtained outside $N(L-1)$, we can use the same argument as in lemma \ref{phi leq 2S, so alsp leq S}: at the maximum point
    			$$0\leq\Delta|\phi|^2=\Delta|V_t\phi|^2\leq-\frac{\kappa}{2}|V_ta|^2-\frac{1}{2}\rho_L^{\tau}|V_t\phi|^4$$
    			implying the bound $\|V_t\phi\|^2_{C^0}\leq S$. 
    		\end{enumerate}
    		
    		To get the sharper bound on $|a|$, since $V_ta$ is harmonic on $N(L-1)$, we can apply the maximum principle to $|V_ta|^2=|a|^2$: $\|a\|_{C^0(N(L-1))}\leq\|a\|_{C^0(X')}=\|a\|_{C^0(X)}.$ Let $\beta_{\pm}$ be smooth cutoff functions supported on $X^{\pm}\setminus N(1)$ with $\beta_{\pm}(x)=1$ for $x\notin X^{\pm}\cap N(L-1)$ and $|d\beta_{\pm}|<\frac{2}{L}$ on $N(L-1)$. Then the function $\nu=1-\beta_--\beta_+$ is supported on $N(L-1)$. The estimate in proposition \ref{phi} gives
    		\begin{align*}
    			\|a\|_{C^0(X')} &\leq C_3(\|(d^*+d^+)\beta_+a\|_{L^p(X)}+\|(d^*+d^+)(\nu+\beta_-)a\|_{L^p(X)})\\
    			&\leq C_3(\|(d^*+d^+)\beta_+a\|_{L^p(X)}+\|(d^*+d^+)\nu a\|_{L^p(X)}+\|(d^*+d^+)\beta_-a\|_{L^p(X)})\\
    			&\leq C_3(\frac{1}{2}S+\|F_A^+\|_{C^0})Vol(X')^{\frac{1}{p}}+C_3(\|d\beta_{-}\|_{L^p(X)}+\|d\nu\|_{L^p(X)}+\|d\beta_{+}\|_{L^p(X)})\|a\|_{C^0(X')}\\
    			&\leq\frac{1}{2}U+8C_3L^{\frac{1}{p}-1}\operatorname{Vol}(S^3)^{\frac{1}{p}}(2U).
    		\end{align*}
    		The second summand can me made smaller than $\frac{1}{2}U$ if we take $L$ sufficiently large. Hence we obtain the desired bound $\|a\|_{C^0(X)}\leq U$.
    	\end{proof}
    \end{lemma} 
    
    Proof of Theorem \ref{gluing isomorphism}: The first two homotopies combine to give a homtopy $\mu\sim P$. The third homotopy is between $P$ and $V^{-1}P^{\tau}V$. Conjugate the first two homotopies by $V$ to get the homotopy $V^{-1}\mu_{2-t}^{\tau}V$ from $V^{-1}P^{\tau}V$ to $V^{-1}\mu^{\tau}V$. The composition of these homotopies gives the desired commutativity of the diagram \ref{homotopy commutative diagram}. 
   
	%\section{Relation with the Bauer-Furuta invariant on the double cover $\tilde{X}$}
	
	%Let $\bar{\mu}$ be the monopole map which restricted to $\mu$ on $X$
	
	%$Z_2$ equivariance of $[\mu]$ and $S^1$ equivariance of $\bar{\mu}$. We may restrict $\bar{\mu}$ to the $Z_2$ part to get someknowledg on $[\mu]$.

	%\section{More refined stable cohomotopy invariant in the case $\tilde{c}_1(X)=0$}
	%$\Z_4$ equivariance for $[\mu]$ of $X$ and $Pin(2)$ equivariance for $\bar{\mu}$ on the spin $\tilde{X}$
		
	%begin{theorem}
	%	As an application of the connected sum formula, let us compute the $Pin^{-}(2)$ Bauer-Furuta invariant of $X=K\# N$, where $N$ is an Enriques surface and $K$ is a $K3$ surface.
		
	%	$[\mu_{N}]\in \pi^2_{Z_2}(\tilde{\R}^2)$ is represented by a map $f:(\R^m\oplus\tilde{\R}^{n+2})^+\to (\R^{m+2}\oplus\tilde{\R}^{n})^+$ which is a generator. The invariant for a $K3$ surface $X$ is $[\mu_X]\in\pi^3_{S_1}(\C^2)$, represented by $g:(\R^t\oplus\C^{s+2})^+\to (\R^{t+3}\oplus\C^{s})^+$, which is also a generator. Nonequivariantly, $g$ represents the Hopf map, and the restriction $\Res^{S^1}_{\Z_2}[\mu_X]$ is the generator of $H^{3}(\R P^3;\pi^{-1}_{st}(*))=\Z_2$. It turns out that the product element $[f]\wedge \Res^{S^1}_{\Z_2}[\mu_X]\in \pi^5_{\Z_2}(\tilde{\R}^6)$ is the genrator of $\Z_2\cong H^5(\R P^5;\pi^{-1}_{st}(*))\subset\pi^4_{st}(\R P^5)$, which maps to $0$ via the comparision map. 
	%\end{theorem}
	
	\begin{example}\label{connected sum of K and N}
		As an application of the connected sum formula, consider $X=K\# N$, where $N$ is an Enriques surface and $K$ is a $K3$ surface. By Theorem \ref{Pin(2) gluing formula}, the $Pin^{-}(2)$ monopole invariant $SW^{Pin}(X)$ vanishes. The invariant $[\mu_N] \in \pi^2_{\Z_2}(\tilde{\R}^2)$ is represented by a generator $f:(\tilde{\R}^2)^+\to (\R^2)^+$. The Bauer-Furuta invariant $[\mu_K]\in\pi^3_{S_1}(\C^2)$ of $K$ is represented by a generator $g:(\R\oplus\C^2)^+\to(\R^4)^+$, which nonequivariantly corresponds to the Hopf map. Using the connected sum formula, the invariant for $X$ is given by the smash product $[\mu_X]=[\mu_N]\wedge\Res^{S^1}_{\Z_2}[\mu_K]\in\pi^5_{\Z_2}(\tilde{\R^6})$, which can be represented by $h=f\wedge\Res^{S^1}_{\Z_2}g:(\R\oplus\tilde{\R^6})^+\to(\R^6)^+$.
	\end{example}
	
	\begin{remark}
		Equivalently, the invariant for $X$ in Example \ref{connected sum of K and N} can be interpreted using the join operation $*$: By restrciting $f$ and $\Res^{S^1}_{\Z_2}g$ to the unit  spheres, we obtain $\overline{f}:S(\tilde{\R^2})\to S(\R^2)$ and $\overline{\Res^{S^1}_{\Z_2}g}:S(\R\oplus\tilde{\R^4})\to S(\R^4)$, hence $\Sigma\left(\overline{f}\wedge\overline{\Res^{S^1}_{\Z_2}g}\right) =\overline{f}*\overline{\Res^{S^1}_{\Z_2}g}:S(\R\oplus\tilde{\R^6})\to S(\R^6)$. Let $\tilde{\cdot}$ denote the corresponding element in the nonequivariant group under the isomorphism given by Proposition \ref{isomorphism between equivariant and nonequivariant}, one may determine the precise stable equivariant cohomotopy class using the following commutative diagram in the $\Z_2$ equivariant setting following the idea of \cite{furuta2007nilpotency}: 
		\begin{figure}[h]
			\centering
			\[
			\begin{tikzcd}\label{}
				\R P^{(2+4-1)} \arrow[r, "\tilde{h}"] \arrow[d, "/{((\Z_2)^2/{\Delta\Z_2})}"'] & S^{(2+3-1)} \arrow[d, "\cong"] \\
				\R P^1*\R P^3 \arrow[r, "\tilde{f}*\Res^{S^1}_{\Z_2}\tilde{g}"] \arrow[d, "\cong"']& S^1*S^3 \arrow[d,"\cong"] \\
				\Sigma(\R P^1\wedge \R P^3) \arrow[r,"\Sigma\tilde{f}\wedge\Res^{S^1}_{\Z_2}\tilde{g}"'] & \Sigma(S^1\wedge S^3)
			\end{tikzcd}
			\]
			\caption{join commutative diagram}
			\label{join commutative}
		\end{figure}
		
		Here $\tilde{h}$ represents the nonequivariant invariant of $K\# N$ and $\Sigma\tilde{f}\wedge\Res^{S^1}_{\Z_2}\tilde{g}$ represents a nontrivial class. The difficulty in determing whether $\tilde{h}$ is trivial or not, is in that the quotient map $/{((\Z_2)^2/{\Delta\Z_2})}:\R P^{(2+4-1)}\to \R P^1 *\R P^3$ is not injective and we cannot conclude that $\tilde{h}$ is nontrivial via the commutative diagram. We leave the problem as future research.
	\end{remark}
	
	%\begin{theorem}
	%	Let $(K,c_K)$ be a $K3$ surface with the canonical $Spin^{c}$ structure induced from the $Spin$ structure. For $(X,c)=(N\#2K,c_N\#2c_K)$ with the connected sum $Spin^{c-}$ structure, where $c_N$ is the structure as described in example \ref{classical example Enriques surface}, $b_1(X;l)=0, b^+(X;l)=8$ and $d=10=2\times 5$, the stable cohomotopy invariant $[\mu(X,c)]\in \pi^8_{\Z_2,H}(*;\tilde{\R}^{10})\cong\pi^7(\R P^9)$ is trivial since the group is trivial.
	%\end{theorem}

	\begin{remark}
		The example above demonstrates that there may exist manifolds whose $Pin^{-}(2)$-monopole class lies in the kernel of the Hurewicz map, which suggests the stable cohomotopy invariants are refinements of the $Pin^{-}(2)$-monopole invariants.
	\end{remark}
	
	%Consider $S^2\times S^2\# (S^2\times S^2)/I)$ with $b_1=0,b^+(X;l)=2$ and $ind(D)=3$, which is a candicate for the generator of $\pi^2_{Z_2}(\tilde{\R}^3)\cong \Z_2$ ($H^1(\R P^2;\Z)=0$). 

	%$E(N)\#_k W_i \# _l Z_j$ for $W_i=\Sigma_i\times S^2$, $g(\Sigma_i)\geq 1$, $Z_j=Y_j\times S^1$, $Y_i$ a closed $3$-manifold.
	
	%\section{$Pin^{-}(2)$ SWF}
	
	\section{Problems and Future Directions}\label{Problems and Future Directions}
	Several potential problems for further work remain:
	\begin{problem}
		Suppose $(Y,c)$ is a $3$-dimensional, oriented, closed manifold with twisted  $Spin^{c-}$ structure $c$ and $b_1(Y;l)=0$. Can we develop the corresponding  Seiberg-Witten-Floer stable homotopy type of $(Y,c)$ following the spirit of \cite{manolescu2003seiberg}?
	\end{problem}
	\begin{problem}
		For the connected sum formula, the thesis only establishes the case where both $b_1(X_1;l_1)$ and $b_2(X_2)$ vanish. What about if one of them, say, $b_1(X_1;l_1)\geq 1?$ Moreover, when both $Spin^{c-}$ structures are twisted, we have $b_1(X_1\#X_2;l_1\#l_2)=b_1(X_1;l_1)+b_2(X_2;l_2)+1$, so there exists essential difficulty in formulating the correct connected sum formula in this case. What kind of tools or theory should we develop to resolve this chanllenge?
	\end{problem}
	
	\bibliography{refers}
	%\nocite{*}
	
	\vspace{1em}
	
	\textsc{Department of Mathematics, Graduate School of Science, Kyoto University, Sakyoku, Kyoto 606–8502, Japan}\\
	
	\textit{Email address}: wu.hao.62w@st.kyoto-u.ac.jp
\end{document}